\documentclass[11pt, reqno]{amsart}

\usepackage{amsmath}
\usepackage{amssymb}
\usepackage{amsthm}
\usepackage{mathrsfs}
\usepackage{mathtools}
\usepackage{bm}
\usepackage{comment}
\usepackage{color}
\usepackage{listings}

\usepackage[dvipdfmx, hypertexnames=false, hidelinks]{hyperref}
\usepackage{cleveref}
\usepackage{autonum}

\hypersetup{final}

\allowdisplaybreaks[3]
\theoremstyle{plain} 
\newtheorem{theorem}{Theorem}
\newtheorem{lemma}{Lemma}[section]
\newtheorem{corollary}{Corollary}

\newtheorem{proposition}{Proposition}

\newtheorem*{conjecture*}{Conjecture}
\newtheorem*{theorem*}{Theorem}
\newtheorem*{question*}{Question}

\theoremstyle{plain}

\theoremstyle{remark}

\theoremstyle{definition}
\newtheorem*{assumption*}{Assumption}

\newtheorem*{notations*}{Notations}
\newtheorem*{acknowledgment*}{Acknowledgments}

\numberwithin{equation}{section}

\crefname{section}{Section}{Sections}

\crefname{theorem}{Theorem}{Theorems}
\crefname{corollary}{Corollary}{Corollaries}
\crefname{lemma}{Lemma}{Lemmas}
\crefname{proposition}{Proposition}{Propositions}
\crefname{claim}{Claim}{Claims}
\crefname{definition}{Definition}{Definitions}
\crefname{notation}{Notation}{Notations}
\crefname{problem}{Problem}{Problems}
\crefname{note}{Note}{Notes}
\crefname{remark}{Remark}{Remarks}
\crefname{example}{Example}{Examples}
\crefname{equation}{}{}

\crefname{enumi}{}{}
\crefname{enumii}{}{}
\crefname{enumiii}{}{}

\newcommand\swapcommand[2]{%
\let\swaptemp#1
\let#1#2
\let#2\swaptemp
}

\let\sl\l
\renewcommand\l{%
		\leavevmode
	\ifmmode
	\left
	\else
		\sl
	\fi
}

\let\sL\L
\renewcommand\L{%
		\leavevmode
	\ifmmode
	\mathscr{L}
	\else
		\sL
	\fi
}

\swapcommand{\SS}{\S}
\renewcommand{\S}{\mathscr{S}}

\newcommand\set[2]{%
	\left\{ #1:  #2 \right\}
}

\newcommand\EXP[1]{
	\mathbb{E}\left[ #1 \right]
}

\newcommand\GS[1]{
	\lfloor #1 \rfloor
}

\newcommand{\CC}{\mathbb{C}}
\newcommand{\RR}{\mathbb{R}}

\newcommand{\ZZ}{\mathbb{Z}}

\newcommand{\PP}{\mathbb{P}}
\newcommand{\e}{\varepsilon}
\newcommand{\s}{\sigma}

\newcommand{\lam}{\lambda}
\newcommand{\Lam}{\Lambda}

\newcommand{\qquadt}{\qquad\qquad\qquad}
\newcommand{\qquadf}{\qquad\qquad\qquad\qquad}

\newcommand{\Sc}{\mathcal{S}}
\newcommand{\ol}{\overline}

\newcommand{\I}{\mathcal{I}}
\newcommand{\ceq}{\coloneqq}
\newcommand{\eqc}{\eqqcolon}

\newcommand{\meas}{\operatorname{meas}}

\renewcommand{\b}{\beta}
\renewcommand{\i}{\mathrm{i}}
\renewcommand{\k}{\kappa}
\renewcommand{\r}{\right}

\renewcommand{\Re}{\operatorname{Re}}
\renewcommand{\Im}{\operatorname{Im}}
\renewcommand{\epsilon}{\varepsilon}

\renewcommand{\tilde}{\widetilde}

\title[Exponential moments of $\Re e^{- \mathrm{i} \theta} \log{\zeta(\tfrac{1}{2} + \mathrm{i} t)}$]
{Exponential moments of the logarithm of the Riemann zeta-function twisted by arguments}
\author[S. Inoue]{Sh\={o}ta Inoue}

\address{Department of Mathematics, Tokyo Institute of Technology, 2-12-1 Ookayama, Meguro-ku, Tokyo 152-8551, Japan}
\email{inoue.s.bd@m.titech.ac.jp}



\begin{document}

\begin{abstract}
	We discuss moments of the Riemann zeta-function in this paper.
	The purpose of this paper is to give an upper bound of exponential moments of the logarithm of the Riemann zeta-function twisted by arguments.
	Our results contain an improvement of Najnudel result for exponential moments of the argument of the Riemann zeta-function
	and an unconditional upper bound of the moments.
\end{abstract}

\maketitle

\section{\textbf{Introduction and statement of the main result}}

	\subsection{Introduction}
		Denote moments of the Riemann zeta-function $\zeta$ on the critical line by
		\begin{align}
			M_{k}(T) \ceq \int_{T}^{2T}|\zeta(\tfrac{1}{2} + \i t)|^{2k} dt
		\end{align}
		for $k > -\frac{1}{2}$ and $T \geq 3$.
		The study of $M_{k}(T)$ is one of the main topics in analytic number theory.
		A motivation of the study is the application to the distribution of zeros and prime numbers.
		For example, Ingham \cite{In1937} showed that the estimate $M_{k}(T) \ll_{\e, k} T^{1 + \e}$ for any $\e > 0$, $k \in \ZZ_{> 0}$,
		which is equivalent to the Lindel\"of Hypothesis,
		would imply the Density Hypothesis and $p' - p \ll_{\e} p^{1/2 + \e}$ for any pair of consecutive prime numbers $(p, p'), p < p'$.
		Keating and Snaith \cite{KS2000} conjectured that
		$M_{k}(T) \sim C(k) T (\log{T})^{k^{2}}$ gives the exact asymptotic behavior for $k > -\frac{1}{2}$
		with $C(k)$ an explicitly written constant.
		This conjecture has been proved in the cases $k = 1, 2$ by Hardy and Littlewood \cite{HL1918}, and by Ingham \cite{In1926} respectively.
		However, the other cases are unproven.
		Toward the conjecture, Heap, Radziwi\l\l, and Soundararajan \cite{HRS2019} established $M_{k}(T) \ll T(\log{T})^{k^{2}}$ for $0 \leq k \leq 2$.
		In the case $k > 2$, Soundararajan \cite{SM2009} succeeded in giving the great result that $M_{k}(T) \ll_{k, \e} T (\log{T})^{k^{2} + \e}$
		for any $k \geq 0$ under the Riemann Hypothesis.
		Moreover, as a celebrated result, Harper \cite{H2013} improved Soundararajan's estimate to $M_{k}(T) \ll_{k} T (\log{T})^{k^{2}}$
		for any $k \geq 0$.

		Recently, Najnudel \cite{Na2020} studied the exponential moments of the argument of the Riemann zeta-function on the critical line, namely
		\begin{align}
			\tilde{M}_{k}(T)
			= \int_{T}^{2T}\exp\l( 2k \Im \log{\zeta(\tfrac{1}{2} + \i t)} \r)dt.
		\end{align}
		This is an analog of $M_{k}(T)$ since replacing $\Im$ with $\Re$ matches $M_{k}(T)$.
		The study of $\tilde{M}_{k}(T)$ is also of interest independent of $M_{k}(T)$ because
		the argument of the Riemann zeta-function is directly related to the distribution of the imaginary parts of nontrivial zeros of $\zeta$.
		The relation is explained by the Riemann-von Mangoldt formula:
		\begin{align}
			N(T)
			= \frac{T}{2\pi}\log{\frac{T}{2\pi}} - \frac{T}{2\pi} + \frac{1}{\pi} \Im \log{\zeta(\tfrac{1}{2} + \i T)} + \frac{7}{8} + O\l( \frac{1}{T} \r).
		\end{align}
		Here, $N(T)$ is the number of nontrivial zeros $\rho = \b + \i \gamma$ of $\zeta$ with $0 < \gamma < T$ counted with multiplicity.
		Najnudel showed that $\tilde{M}_{k}(T) \ll_{k, \e} T(\log{T})^{k^{2} + \e}$ for $k \in \RR$ under the Riemann Hypothesis.
		His upper bound is an analog of Soundararajan's.
		The analog of Harper's bound for $\tilde{M}_{k}(T)$ has not proven yet.
		Additionally, the methods for $M_{k}(T)$ in \cite{HL1918}, \cite{HRS2019}, \cite{In1926} do not work for $\tilde{M}_{k}(T)$,
		so it also remains the problem of giving unconditionally a good estimate of $\tilde{M}_{k}(T)$, for example,
		$\tilde{M}_{k}(T) \ll T^{1 + \e}$ for $0 \leq k \leq c$ with $c$ a certain absolute constant.
		The purpose of this paper is to discuss these problems.

		To discuss our problems more generally, we define the exponential moment
		\begin{align}
			M_{k, \theta}(T)
			\ceq \int_{T}^{2T}\exp\l( 2k \Re e^{- \i \theta} \log{\zeta(\tfrac{1}{2} + \i t)} \r)dt
		\end{align}
		for $k \geq 0$, $\theta \in \RR$, and large $T$.
		It then holds that $M_{k, 0}(T) = M_{k}(T)$ and $M_{k, \frac{\pi}{2}}(T) = \tilde{M}_{k}(T)$,
		so $M_{k, \theta}(T)$ is a generalization of both $M_{k}(T)$ and $\tilde{M}_{k}(T)$.
		This moment has been studied by the author and Li in \cite{IL2021}.
		We proved unconditionally for $0 \leq k \leq c$ that
		$M_{k, \theta}(T) \ll_{k} T (\log{T})^{k^{2} + B k^{3}}$ if $\theta \in [-\frac{\pi}{2}, \frac{\pi}{2}]$,
		and that $M_{k, \theta}(T) \gg_{k} T (\log{T})^{k^{2} - B k^{3}}$ if $\theta \in [\frac{\pi}{2}, \frac{3\pi}{2}]$,
		where $c$ and $B$ are certain absolute positive constants.
		Additionally, assuming the Riemann Hypothesis, we showed for any $k \geq 0$, $\e > 0$ that
		$M_{k, \theta}(T) \ll_{k, \e} T (\log{T})^{k^{2} + \e}$ and $M_{k, \theta}(T) \gg_{k, \e} T (\log{T})^{k^{2} - \e}$
		for the same $\theta$ as the unconditional case.
		In this paper, we address our problems by improving our previous result.

	\subsection{Statement of the main result}
		We state the main result under the following zero density estimate
		\begin{align}
			\label{AZDE}
			N(\s, T) \ll T^{1 - \lam(\s - \frac{1}{2})} \Phi(T), \quad \s \geq \frac{1}{2} + \frac{1}{\log{T}}.
		\end{align}
		Here, $\Phi(T)$ is a nonnegative-valued function independent of $\s$,
		and $N(\s, T)$ is the number of nontrivial zeros $\rho = \b + \i \gamma$
		of the Riemann zeta-function with $\b > \s$, $0 \leq \gamma \leq T$ counted with multiplicity.
		The Density Hypothesis states that \cref{AZDE} $\lam = 2$ and $\Phi(T) \ll_{\e} T^{\e}$ would hold.
		We can also take $\lam = +\infty$ for any $\Phi(T)$ under the Riemann Hypothesis.
		Various unconditional estimates of $N(\s, T)$ have been also proved.
		Ingham \cite{In1940} showed that \cref{AZDE} holds for $\lam = \frac{4}{3}$ and $\Phi(T) = (\log{T})^{5}$.
		When $\Phi(T) = \log{T}$, this type estimate was firstly shown by Selberg \cite{SCR}.
		He proved that \eqref{AZDE} holds with $\lam = \frac{1}{4}$ and $\Phi(T) = \log{T}$.
		Conrey \cite{Co1989} improved Selberg's constant to $\lam = \frac{8}{7} - \e$ with $\e$ any fixed positive constant.
		Throughout this paper, we assume $\Phi(T) \geq \log{T}$, and then may also assume that $\lam \geq 1$ thanks to Conrey's result.

		Define the numbers $a(h), b(h)$ by
		\begin{align}
			\label{def_a_h}
			a(h)
			= \left\{
			\begin{array}{cl}
				1 & \text{if\; $0 < h \leq \frac{1}{\sqrt{2}}$,} \\
				1 + \frac{(4h^{2} - 1)^{2}(1 + h^{2})}{8(4h^{2} + 1)(1 - h^{2})}\log\l( 1 + 2\frac{(2h^{2} - 1)(4h^{2} + 1)}{(h^{2} + 1)(4h^{2} - 1)^{2}} \r)
				& \text{if\; $\frac{1}{\sqrt{2}} < h < 1$,}
			\end{array}
			\right.
		\end{align}
		and
		\begin{align}
			\label{def_b_h}
			b(h)
			= \frac{\pi(1 + h^{2})^{2}}{2(1 - h^{2})}.
		\end{align}
		We then define
		\begin{align}
			\label{def_Atheta}
			A(h, \theta) = \frac{h}{a(h)|\cos{\theta}| + b(h)|\sin\theta|}
		\end{align}
		for $0 < h < 1$, $\theta \in \RR$.
		In the following, let $h = h(\theta)$ be the maximal point of $A(h, \theta)$.
		For example, if $\theta = 0$, then $h = 0.72894\cdots$, $A(h, 0) = 0.70738\cdots$,
		and if $\theta = \pm\frac{\pi}{2}$, then $h = \sqrt{2} - 1$ and $A(h, \pm \frac{\pi}{2}) = \frac{1}{2\pi}$.
		By routine numerical calculations, we can confirm that this $h$ is monotonically non-increasing
		in $\theta \in [0, \frac{\pi}{2}]$, and that
		$\frac{2}{5} < \sqrt{2} - 1 = h(\frac{\pi}{2}) \leq h \leq h(0) = 0.72894\cdots < \frac{3}{4}$ and
		$\frac{1}{9} < \frac{12}{24|\cos \theta| + 25 \pi |\sin \theta|}
			= A(\frac{1}{2}, \theta) \leq A(h, \theta) \leq \frac{h}{|\cos{\theta}| + |\sin{\theta}|} \leq h < 1$
		for any $\theta \in \RR$, so we may assume $\frac{2}{5} < h < \frac{3}{4}$ and $\frac{1}{9} < A(h, \theta) < 1$
		throughout this paper.

		The following is the main result of this paper.

		\begin{theorem} \label{MTUBMZ}
			Let $\lam$ be a positive constant, and $\Phi$ be a nonnegative-valued function satisfying \cref{AZDE}.
			Let $\theta \in [-\tfrac{\pi}{2}, \tfrac{\pi}{2}]$, and let $0 < \e \leq \frac{1}{100}$ be a small constant.
			For any $0 \leq k < A(h, \theta) \lam - \e$
			and any $T \geq e^{C_{1} (k \log{(k + 3)})^{2}} T_{0}$, we have
			\begin{align} \label{MTUBMZ1}
				M_{k, \theta}(T)
				&\leq \exp(C_{1} e^{C_{1} k}) T (\log{T})^{k^{2}}
				+ C_{1} T(\log{T})^{k^{2}(1 + \frac{k + \e}{A(h, \theta) \lam - (k + \e)})}
				(\Phi(T) / \log{T})^{\frac{k + \e}{A(h, \theta) \lam}}\\
				&\qquad + C_{1} T \frac{\Phi(T)}{\log{T}}.
			\end{align}
			Here, $C_{1}$ and $T_{0}$ are positive numbers depending only on $\e$.
		\end{theorem}

		\Cref{MTUBMZ} involves an improvement of the estimate due to Najnudel \cite[Theorem 1]{Na2020},
		and also a small improvement of Harper's result \cite[Theorem 1]{H2013} for the dependence of $T$ and $k$.
		For instance, we can deduce the following corollary from this theorem and the zero density estimate by Ingham.

		\begin{corollary}
			For any small $\e > 0$ and $T \geq T_{1}(\e)$ with $T_{1}(\e)$ a large constant depending only on $\e$, we have
			\begin{align}	\label{ES:UBMKtuc}
				\tilde{M}_{k}(T)
				\ll_{\e} T (\log{T})^{B(k) + \e}
				\ll_{\e} T^{1 + \e}
			\end{align}
			for $|k| \leq c_{0} - \e$ with $c_{0} = \frac{4}{3}A(h, \pm\frac{\pi}{2}) = \frac{2}{3\pi}$
			and $B(k) = k^{2}(1 + k / (c_{0} - k)) + 4k / c_{0}$ unconditionally.
			Assuming the Riemann Hypothesis, we have
			\begin{align}	\label{ES:UBMKt}
				M_{k, \theta}(T)
				\ll e^{e^{O(k)}} T (\log{T})^{k^{2}}
			\end{align}
			for any $\theta \in [-\frac{\pi}{2}, \frac{\pi}{2}]$, $k \geq 0$, and any $T \geq e^{C (k \log(k + 3))^{2}}T_{0}$
			with $C$ and $T_{0}$ large absolute constants.
		\end{corollary}

		The size of the term $e^{e^{O(k)}}$ in \cref{ES:UBMKt} is the same as the result of Harper.
		On the other hand, the dependency of $T$ and $k$ in \cref{ES:UBMKt} allows the range of $k$ to be $k \leq c \sqrt{\log{T} / \log{\log{T}}}$,
		where $c$ is a small constant.
		The original Harper's range is $k \leq c\sqrt{\log{\log{T}}}$, so our range of $k$ is wider than his.
		In particular, choosing $k = \log{\log{T}}$ we obtain Littlewood's conditional bounds
		$\log|\zeta(\frac{1}{2} + \i t)| \leq C\frac{\log{t}}{\log{\log{t}}}$
		and $S(t) \ceq \frac{1}{\pi}\Im \log{\zeta(\tfrac{1}{2} + \i t)} \ll \frac{\log{t}}{\log{\log{t}}}$.
		This is a new application of our result and may be one reason to show the sharpness of the term $e^{e^{O(k)}}$ given by Harper.

	\subsection{Organization of the paper and some comments on the proof of \cref{MTUBMZ}}
		\Cref{MTUBMZ} is proved by combining Harper's method \cite{H2013} with a new approximate formula proved in \cref{Sec:AFLZ}.
		The main part of the proof of \cref{MTUBMZ} is devoted to determine the range of $k$ in the unconditional case.
		Actually, conditional result \cref{ES:UBMKt} can be essentially established by using Harper's method and the approximate formula in \cite{II2019}.
		On the other hand, the range of $k$ in the unconditional case derived from the approximate formula is narrow
		because the constant multiplied by the error term of the formula is big.
		In fact, the range of $k$ given in Theorem 2.4 of \cite{IL2021} proved by using the approximate formula in \cite{II2019}
		is roughly $0 \leq k \leq a_{6} = e^{-10}$.
		For this reason, we prove a new approximate formula for the Riemann zeta-function in this paper.
		The method of the proof is based on Selberg's \cite{SCR}.
		An important change from his method is that the implicit constant is optimized
		by parametrizing the ``zero density function'' defined by \cref{def_s_X}.

		In \cref{Sec:UBDRZ}, we show a result for large deviations of the distribution of the Riemann zeta-function under zero density estimate \cref{AZDE}.
		Harper used the result of moments due to Soundararajan in \cite{SM2009} when evaluating an integral.
		The argument is under the Riemann Hypothesis.
		Therefore, we need to prove an alternative result, which can be used only under the assumption of the zero density estimate.
		We give such a result in \cref{Sec:UBDRZ}.

		In \cref{Sec:EMDP}, we apply Harper's method to our approximate formula.
		Harper used the inequality of Soundararajan \cite{SM2009}, which holds under the Riemann Hypothesis.
		Therefore, we again face the issue of modifying Harper's method appropriately.
		We address this issue by dividing $[T, 2T]$ according to the value of the zero density function,
		in addition to his original method of dividing the length of Dirichlet polynomials.
		However, the problem remains in this method that the power of $\log{T}$ in \cref{ES:UBMKtuc} does not reach $k^{2}$.
		We explain below why this problem arises.
		The sequence of functions $\{p^{-\i t}\}$ on $t \in [T, 2T]$ over prime numbers is convergent weakly as $T \rightarrow \infty$
		to a sequence of independent random variables,
		so $\sum_{p \in \mathcal{P}_{1}}a(p)p^{-\i t}$ and $\sum_{p \in \mathcal{P}_{2}}b(p)p^{-\i t}$
		are also convergent weakly to independent random variables
		for any complex sequences $\{a(p)\}, \{b(p)\}$ and disjoint finite sets of prime numbers $\mathcal{P}_{1}, \mathcal{P}_{2}$.
		This fact plays an important role in the original method of Harper,
		which allows us the upper bound of moments to reach $T (\log{T})^{k^{2}}$ under the Riemann Hypothesis.
		On the other hand, we do not assume the Riemann Hypothesis, so objects written by some combinations of primes and zeros appear.
		However, we do not know an independency between primes and zeros.
		Hence, we cannot deal with objects combining prime factors and zero factors well unconditionally.
		In this paper, we use H\"older's inequality to such objects, but the inequality is not enough to derive the sharp upper bound.
		This is the reason why estimate \cref{ES:UBMKtuc} does not reach $T(\log{T})^{k^{2}}$.
		This problem would be one motivation of the study of the independency between primes and zeros, which is related to the splitting conjecture
		of Gonek, Hughes, and Keating \cite{GHK2007}.

\section{\textbf{An approximate formula for the Riemann zeta-function}}	\label{Sec:AFLZ}

	Let $\rho = \b + \i \gamma$ be a nontrivial zero of the Riemann zeta-function with $\b, \gamma \in \RR$.
	Denote the von Mangoldt function by $\Lam(n)$.
	Define $s_{\theta, K}(t, X) = \s_{\theta, K}(t, X) + \i t$, $w_{X}(y)$, and $P_{\theta, K}(t, X)$ by
	\begin{gather}
		\label{def_s_X}
		\s_{\theta, K}(t, X)
		= \frac{1}{2} + \frac{1}{h} \max_{|t - \gamma| \leq K X^{|\b - 1/2|} / \log{X}}\l\{ \b - \frac{1}{2}, \frac{K}{\log{X}} \r\},\\
		\label{def_w_X}
		w_{X}(y)
		= \l\{
		\begin{array}{cl}
			1                                                                    & \text{if\; $1 \leq y \leq X^{1 / 3}$,}         \\[2mm]
			\frac{9(\log(X / y))^2 - 6(\log(X^{2/3} / y))^2}{2(\log{X^{2/3}})^2} & \text{if\; $X^{1 / 3} \leq y \leq X^{2 / 3}$,} \\[2mm]
			\frac{9(\log(X / y))^2}{2(\log{X})^2}                                & \text{if\; $X^{2 / 3} \leq y \leq X$,}         \\[2mm]
			0                                                                    & \text{otherwise.}
		\end{array}
		\r.\\
		\label{def_PP}
		P_{\theta, K}(t, X) =
		\sum_{p \leq X}\l( \frac{w_{X}(p)}{p^{s_{\theta, K}(t, X)}}\l( 1 + (\s_{\theta, K}(t, X) - \tfrac{1}{2})\log{p} \r)
		+ \frac{1}{2p^{2s_{\theta, K}(t, X)}} \r)
	\end{gather}
	for $\theta \in \RR$, $K \geq 1$, and $X \geq 3$.
	If there are no zeros such that $|t - \gamma| \leq K X^{|\b - 1/2|} / \log{X}$, we let $\s_{\theta, K}(t, X) = \frac{1}{2} + \frac{K}{h \log{X}}$,
	and so the inequality $\s_{\theta, K}(t, X) \geq \frac{1}{2} + \frac{K}{h \log{X}}$ always holds.
	We also define $Y_{\theta, K}(t, X)$ by
	\begin{align}
		\label{def_Y_hK}
		Y_{\theta, K}(t, X) =
		\frac{1}{2}\sum_{\rho \in S}\log\l( \frac{(\s_{\theta, K}(t, X) - \b)^2 + (t - \gamma)^{2}}{(\b - \frac{1}{2})^{2} + (t - \gamma)^{2}} \r),
	\end{align}
	where $S$ is the set of nontrivial zeros $\rho = \b + \i\gamma$ of the Riemann zeta-function such that
	$|\b - \frac{1}{2}| \leq \sqrt{\frac{1}{2}(\s_{\theta, K}(t, X) - \frac{1}{2})^{2} + (t - \gamma)^{2}}$.
	Then we have the following proposition.

	\begin{proposition} \label{AFLZ}
		Let $\theta \in \RR$, and let $K$ be large.
		For any $|t| \geq 3$ with $t$ not equal to the ordinates of zeros and $3 \leq X \leq |t|^{6}$, we have
		\begin{align}
			&\bigg|\Re e^{-\i \theta}\log{\zeta(\tfrac{1}{2} + \i t)} - \Re e^{-\i \theta}P_{\theta, K}(t, X) + \Re e^{-\i\theta}Y_{\theta, K}(t, X) \bigg|\\
			&\leq \l( \frac{h}{A(h, \theta)} + \frac{9}{K} \r)
			(\s_{\theta, K}(t, X) - \tfrac{1}{2})\l( \frac{1}{2}\log{|t|} - \Re \sum_{p \leq X}\frac{w_{X}(p) \log{p}}{p^{s_{\theta, K}(t, X)}} + C_{1}\log{X} \r),
		\end{align}
		and $Y_{\theta, K}(t, X) \geq 0$.
		Here, $C_{1}$ is an absolute positive constant.
	\end{proposition}

	\subsection{Preliminaries}
		We require some lemmas to prove \cref{AFLZ}.

		\begin{lemma}	\label{FLDZ}
			Let $u: \RR_{\geq 0} \rightarrow \RR_{\geq 0}$ be a Riemann-integrable function such that $\int_{0}^{\infty}u(x)dx = 1$,
			and $\tilde{u}(z) = \int_{0}^{\infty}u(x) x^{z - 1}dx$ is absolutely convergent for $\Re z > 0$ and satisfies $\tilde{u}(z) \ll |z|^{-1}$
			in $\Re z \leq C$ for any fixed constant $C > 0$. Put $w_{u, X}(y) = \int_{y^{1 / \log{X}}}^{\infty}u(x)dx$.
			For any $X \geq 3$ and any $s \in \CC$ not equal to the zeros of the Riemann zeta-function, we have
			\begin{align}
				\label{FLDZ1}
				\frac{\zeta'}{\zeta}(s)
				&= -\sum_{n = 1}^{\infty}\frac{\Lam(n)}{n^{s}}w_{u, X}(n) + \sum_{\rho}\frac{1}{s - \rho}\tilde{u}(1 - (s - \rho) \log{X})\\
				&\qquad+ \sum_{n = 1}^{\infty}\frac{1}{s + 2n}\tilde{u}(1 - (s + 2n) \log{X}) - \frac{1}{s - 1}\tilde{u}(1 - (s - 1) \log{X}).
			\end{align}
		\end{lemma}

		\begin{proof}
			By Mellin's transform, we find that
			\begin{align}
				\sum_{n = 1}^{\infty}\frac{\Lam(n)}{n^{s}}w_{u, X}(n)
				&= \sum_{n = 1}^{\infty}\frac{\Lam(n)}{n^{s}}\int_{\log{X} - \i \infty}^{\log{X} + \i \infty}\frac{\tilde{u}(z + 1)}{z} n^{-z / \log{X}}dz\\
				\label{pFLDZ1}
				&= -\int_{c - \i \infty}^{c + \i \infty}\frac{\zeta'}{\zeta}\l(s + \frac{z}{\log{X}}\r)\frac{\tilde{u}(z + 1)}{z} n^{-z / \log{X}}dz
			\end{align}
			with $c = \max\{(2 - \s)\log{X}, 1\}$.
			Note that the absolute convergence of the above integrals is guaranteed by
			$\tilde{u}(z) \ll |z|^{-1}$ as $|\Im z| \rightarrow + \infty$.
			By the residue theorem, \cref{pFLDZ1} is equal to
			\begin{multline}
				-\frac{\zeta'}{\zeta}(s) + \sum_{\rho}\frac{1}{s - \rho}\tilde{u}(1 - (s - \rho) \log{X})\\
				+ \sum_{n = 1}^{\infty}\frac{1}{s + 2n}\tilde{u}(1 - (s + 2n) \log{X}) - \frac{1}{s - 1}\tilde{u}(1 - (s - 1) \log{X}),
			\end{multline}
			which completes the proof of this lemma.
		\end{proof}

		\begin{lemma}	\label{lem_h_zf}
			Let $\theta \in \RR$, and let $K$ be large.
			For any $t \in \RR$, $X \geq 3$, we have
			\begin{align}
				\sum_{\rho}\frac{\s_{\theta, K}(t, X) - \frac{1}{2}}{(\s_{\theta, K}(t, X) - \b)^{2} + (t - \gamma)^{2}}
				\leq \frac{1 + h^{2}}{1 - h^{2}}\sum_{\rho}\frac{\s_{\theta, K}(t, X) - \b}{(\s_{\theta, K}(t, X) - \b)^{2} + (t - \gamma)^{2}}.
			\end{align}
		\end{lemma}

		\begin{proof}
			We use Selberg's method (cf. pp.\ 235--236 in \cite{SCR}).
			By the symmetry of nontrivial zeros, we have
			\begin{align}
				&\sum_{\rho}\frac{\s_{\theta, K}(t, X) - \b}{(\s_{\theta, K}(t, X) - \b)^{2} + (t - \gamma)^{2}}\\
				&= \frac{1}{2}\sum_{\rho}\l( \frac{\s_{\theta, K}(t, X) - \b}{(\s_{\theta, K}(t, X) - \b)^{2} + (t - \gamma)^{2}}
				+ \frac{\s_{\theta, K}(t, X) - (1 - \b)}{(\s_{\theta, K}(t, X) - (1 - \b))^{2} + (t - \gamma)^{2}} \r)\\
				&= \sum_{\rho}\frac{(\s_{\theta, K}(t, X) - \tfrac{1}{2})\l\{(\s_{\theta, K}(t, X) - \frac{1}{2})^{2} - (\b - \frac{1}{2})^{2} + (t - \gamma)^{2}\r\}}{
					\{(\s_{\theta, K}(t, X) - \b)^{2} + (t - \gamma)^{2}\}\{(\s_{\theta, K}(t, X) - (1 - \b))^{2} + (t - \gamma)^{2}\}}.
			\end{align}
			When $|\b - \frac{1}{2}| \leq h (\s_{\theta, K}(t, X) - \frac{1}{2})$, we find by simple calculations that
			\begin{align}
				&(\s_{\theta, K}(t, X) - \tfrac{1}{2})^{2} - (\b - \tfrac{1}{2})^{2} + (t - \gamma)^{2}\\
				&\geq \frac{1 - h^{2}}{1 + h^{2}}\l\{(\s_{\theta, K}(t, X) - \tfrac{1}{2})^{2} + (\b - \tfrac{1}{2})^{2}\r\} + (t - \gamma)^{2}\\
				&\geq \frac{1 - h^{2}}{2(1 + h^{2})}\l\{2(\s_{\theta, K}(t, X) - \tfrac{1}{2})^{2} + 2(\b - \tfrac{1}{2})^{2} + 2(t - \gamma)^{2} \r\}\\
				&= \frac{1 - h^{2}}{2(1 + h^{2})}\l\{ (\s_{\theta, K}(t, X) - \b)^{2} + (\s_{\theta, K}(t, X) - 1 + \b)^{2} + 2(t - \gamma)^{2} \r\}.
			\end{align}
			When $|\b - \frac{1}{2}| > h (\s_{\theta, K}(t, X) - \frac{1}{2})$, it follows from the definition of $\s_{\theta, K}(t, X)$ that
			\begin{align}
				|t - \gamma| > K \frac{X^{|\b - 1/2|}}{\log{X}} > K |\b - \tfrac{1}{2}|.
			\end{align}
			Therefore, we also find by a simple calculation that
			\begin{align}
				&(\s_{\theta, K}(t, X) - \tfrac{1}{2})^{2} - (\b - \tfrac{1}{2})^{2} + (t - \gamma)^{2}\\
				&\geq \frac{1 - 2 / K^{2}}{2} \l\{ (\s_{\theta, K}(t, X) - \b)^{2} + (\s_{\theta, K}(t, X) - 1 + \b)^{2} + 2(t - \gamma)^{2} \r\}
			\end{align}
			when $|\b - \frac{1}{2}| > h (\s_{\theta, K}(t, X) - \frac{1}{2})$.
			Hence, we have
			\begin{align}
				&\sum_{\rho}\frac{\s_{\theta, K}(t, X) - \b}{(\s_{\theta, K}(t, X) - \b)^{2} + (t - \gamma)^{2}}\\
				&\geq \frac{1 - h^{2}}{2(1 + h^{2})}\sum_{\rho}\l(\frac{\s_{\theta, K}(t, X) - \tfrac{1}{2}}{
						(\s_{\theta, K}(t, X) - \b)^{2} + (t - \gamma)^{2}}
				+ \frac{\s_{\theta, K}(t, X) - \tfrac{1}{2}}{(\s_{\theta, K}(t, X) - (1 - \b))^{2} + (t - \gamma)^{2}}\r)\\
				&= \frac{1 - h^{2}}{1 + h^{2}}\sum_{\rho}\frac{\s_{\theta, K}(t, X) - \frac{1}{2}}{(\s_{\theta, K}(t, X) - \b)^{2} + (t - \gamma)^{2}},
			\end{align}
			which completes the proof of this lemma.
		\end{proof}

		In the following, we choose $u$ as
		\begin{align}
			u(x) = \l\{
			\begin{array}{cl}
				\frac{9\log{x} - 3}{x} & \text{if \; $e^{1/3} \leq x \leq e^{2/3}$,} \\
				\frac{9 - 9\log{x}}{x} & \text{if \; $e^{2/3} \leq x \leq e$,}       \\
				0                      & \text{otherwise.}
			\end{array}
			\r.
		\end{align}
		We then find that $\tilde{u}(1 - z) = 9(e^{-z/6} - e^{-z/2})^{2} / z^{2}$,
		and that $w_{u, X}(y)$ in \cref{FLDZ} coincides with $w_{X}(y)$ defined by \cref{def_w_X}.
		We should remark that our constant $A(h, \theta)$ does not depend on the choice of $u$ in our method
		as long as $u$ is supported on $[c, e]$	and belongs to $C^{2}([c, e])$ for any fixed $c \in (0, e)$.

		\begin{lemma}	\label{INE_MZ_NZ}
			Let $\theta \in \RR$, and let $K$ be large.
			For $t \in \RR$, $X \geq 3$, $\s \geq \s_{\theta, K}(t, X)$, we have
			\begin{align}
				&\bigg|\sum_{\rho}\frac{1}{s - \rho}\tilde{u}(1 - (s - \rho)\log{X})\bigg|
				\ll X^{-(\s - \s_{\theta, K}(t, X))/3}K^{-4}
				\sum_{\rho}\frac{\s_{\theta, K}(t, X) - \frac{1}{2}}{(\s_{\theta, K}(t, X) - \b)^{2} + (t - \gamma)^{2}}.
			\end{align}
			This implicit constant is absolute.
		\end{lemma}

		\begin{proof}
			From the present definition of $u$, we have
			\begin{align}
				&\bigg|\sum_{\rho}\frac{1}{s - \rho}\tilde{u}(1 - (s - \rho)\log{X})\bigg|
				\leq \frac{9}{(\log{X})^{2}}\sum_{\rho}\frac{(X^{-(\s - \b)/6} + X^{-(\s - \b)/2})^{2}}{((\s - \b)^{2} + (t - \gamma)^{2})^{3/2}}.
			\end{align}
			When $\b - \frac{1}{2} \leq h (\s_{\theta, K}(t, X) - \frac{1}{2})$, we find that
			\begin{align}
				&\frac{(X^{-(\s - \b)/6} + X^{-(\s - \b)/2})^{2}}{((\s - \b)^{2} + (t - \gamma)^{2})^{3/2}}
				\leq \frac{4X^{-(\s - \b)/3}}{((\s_{\theta, K}(t, X) - \b)^{2} + (t - \gamma)^{2})^{3/2}}\\
				&\leq \frac{4X^{-(\s - \s_{\theta, K}(t, X))/3}X^{-(\s_{\theta, K}(t, X) - \b)/3}}{
					(\s_{\theta, K}(t, X) - \b)\l\{(\s_{\theta, K}(t, X) - \b)^{2} + (t - \gamma)^{2}\r\}}\\
				&\leq \frac{4 e^{-(1 - h) K / 3 h} X^{-(\s - \s_{\theta, K}(t, X))/3}}{(1 - h)(\s_{\theta, K}(t, X) - \frac{1}{2})
					\{ (\s_{\theta, K}(t, X) - \b)^{2} + (t - \gamma)^{2} \}}\\
				&\leq X^{-(\s - \s_{\theta, K}(t, X))/3} \frac{(\log{X})^{2}}{K^{4}}\frac{\s_{\theta, K}(t, X) - \frac{1}{2}}{(\s_{\theta, K}(t, X) - \b)^{2} + (t - \gamma)^{2}}
			\end{align}
			for large $K$.
			When $\b - \frac{1}{2} > h (\s_{\theta, K}(t, X) - \frac{1}{2})$, it follows from the definition of $\s_{\theta, K}(t, X)$ that
			\begin{align}
				|t - \gamma| > K \frac{X^{\b - 1/2}}{\log{X}} > K (\b - \tfrac{1}{2}) > \frac{K h}{1 + h}|\s_{\theta, K}(t, X) - \b|,
			\end{align}
			and so
			\begin{align}
				(t - \gamma)^{2} > \l( 1 + \frac{(1 + h)^{2}}{K^{2} h^{2}} \r)^{-1}\l\{ (\s_{\theta, K}(t, X) - \b)^{2} + (t - \gamma)^{2} \r\}.
			\end{align}
			Therefore, we find that
			\begin{align}
				&\frac{(X^{-(\s - \b)/6} + X^{-(\s - \b)/2})^{2}}{((\s - \b)^{2} + (t - \gamma)^{2})^{3/2}}\\
				&\leq \frac{1}{|t - \gamma|}\l( 1 + \frac{(1 + h)^{2}}{K^{2} h^{2}} \r)
				\frac{(X^{-(\s - \b)/6} + X^{-(\s - \b)/2})^{2}}{(\s_{\theta, K}(t, X) - \b)^{2} + (t - \gamma)^{2}}\\
				&\leq \frac{2 \log{X}}{K X^{\b - 1/2}}\frac{(X^{-(\s - \b)/6} + X^{-(\s - \b)/2})^{2}}{
					(\s_{\theta, K}(t, X) - \b)^{2} + (t - \gamma)^{2}}\\
				&\leq 2(X^{-(\s + 2\b - \frac{3}{2})/6} + X^{-(\s - \frac{1}{2})/2})^{2}
				\frac{h (\log{X})^{2}}{K^{2}}\frac{\s_{\theta, K}(t, X) - \frac{1}{2}}{(\s_{\theta, K}(t, X) - \b)^{2} + (t - \gamma)^{2}}\\
				&\ll X^{-(\s - \s_{\theta, K}(t, X))/3}\frac{(\log{X})^{2}}{K^{4}}\frac{\s_{\theta, K}(t, X) - \frac{1}{2}}{(\s_{\theta, K}(t, X) - \b)^{2} + (t - \gamma)^{2}}
			\end{align}
			for large $K$.
			Combing these inequalities, we have
			\begin{align}
				&\bigg|\sum_{\rho}\frac{1}{s - \rho}\tilde{u}(1 - (s - \rho)\log{X})\bigg|
				\ll X^{-(\s - \s_{\theta, K}(t, X))/3}\frac{1}{K^{4}} \frac{\s_{\theta, K}(t, X) - \frac{1}{2}}{
					(\s_{\theta, K}(t, X) - \b)^{2} + (t - \gamma)^{2}},
			\end{align}
			which completes the proof of this lemma.
		\end{proof}

		\begin{lemma}	\label{ERZDP}
			Let $\theta \in \RR$, and let $K$ be large.
			For any $|t| \geq 3$, $3 \leq X \leq |t|^{6}$, we have
			\begin{align}	\label{ERZDP1}
				&\sum_{\rho}\frac{\s_{\theta, K}(t, X) - \b}{(\s_{\theta, K}(t, X) - \b)^{2} + (t - \gamma)^{2}}\\
				&= \l(1 + O\l( \frac{1}{K^{3}} \r)\r)\l( \frac{1}{2} \log{|t|} - \Re \sum_{p \leq X}\frac{w_{X}(p) \log{p}}{p^{s_{\theta, K}(t, X)}} \r) + O(\log{X}).
			\end{align}
		\end{lemma}

		\begin{proof}
			We use the well known formula (see (2.10) in \cite{MV})
			\begin{align}	\label{BFLDRZ}
				\frac{\zeta'}{\zeta}(s)
				= \sum_{\rho}\l( \frac{1}{s - \rho} + \frac{1}{\rho} \r) - \frac{1}{2} \log{|t|} + O(1)
			\end{align}
			for $0 \leq \s \leq 2$, $|t| \geq 1$ to obtain
			\begin{align}
				\Re\frac{\zeta'}{\zeta}(s_{\theta, K}(t, X))
				= \sum_{\rho}\frac{\s_{\theta, K}(t, X) - \b}{(\s_{\theta, K}(t, X) - \b)^{2} + (t - \gamma)^{2}} - \frac{1}{2}\log{|t|} + O(1).
			\end{align}
			On the other hand, we also deduce from the definition of $u$ and \cref{FLDZ,INE_MZ_NZ} that
			\begin{align}
				&\Re\frac{\zeta'}{\zeta}(s_{\theta, K}(t, X))\\
				&= -\Re\sum_{n \leq X}\frac{\Lam(n) w_{X}(n)}{n^{s_{\theta, K}(t, X)}}
				+ \omega\sum_{\rho}\frac{\s_{\theta, K}(t, X) - \b}{(\s_{\theta, K}(t, X) - \b)^{2} + (t - \gamma)^{2}}
				+ O\l( \frac{X^{1/2}}{|t|^{3}} + 1 \r)
			\end{align}
			for some $\omega = \omega(t, X, \theta, K) \in \RR$ with $|\omega| \leq K^{-3}$ for any large $K$.
			Combining these two formulas, we obtain
			\begin{align}
				&\l( 1 - \omega \r)\sum_{\rho}\frac{\s_{\theta, K}(t, X) - \b}{(\s_{\theta, K}(t, X) - \b)^{2} + (t - \gamma)^{2}}
				= \frac{1}{2}\log{|t|} - \Re\sum_{n \leq X}\frac{\Lam(n) w_{X}(n)}{n^{s_{\theta, K}(t, X)}} + O\l( 1 \r).
			\end{align}
			Moreover, we find, using the prime number theorem and partial summation, that
			\begin{align}
				\sum_{n \leq X}\frac{\Lam(n) w_{X}(n)}{n^{s_{\theta, K}(t, X)}}
				= \sum_{p \leq X}\frac{w_{X}(p) \log{p}}{p^{s_{\theta, K}(t, X)}} + O(\log{X}).
			\end{align}
			Hence, we obtain
			\begin{align}
				&\sum_{\rho}\frac{\s_{\theta, K}(t, X) - \b}{(\s_{\theta, K}(t, X) - \b)^{2} + (t - \gamma)^{2}}\\
				&= \l( 1 + O\l( \frac{1}{K^{3}} \r) \r)\l( \frac{1}{2}\log{|t|} - \Re\sum_{n \leq X}\frac{\Lam(n) w_{X}(n)}{n^{s_{\theta, K}(t, X)}} \r)
				+ O\l( \log{X} \r),
			\end{align}
			which completes the proof of \cref{ERZDP}.
		\end{proof}

	\subsection{Proof of \cref{AFLZ}}
		We start from the formula
		\begin{align}
			\log\zeta(\tfrac{1}{2} + \i t)
			&= \int_{+\infty}^{\frac{1}{2}}\frac{\zeta'}{\zeta}(\s + \i t)d\s\\
			&=\int_{+\infty}^{\s_{\theta, K}(t, X)}\frac{\zeta'}{\zeta}(\s + \i t)d\s
			+ (\tfrac{1}{2} - \s_{\theta, K}(t, X))\frac{\zeta'}{\zeta}(\s_{\theta, K}(t, X) + \i t)\\
			&\quad+ \int_{\s_{\theta, K}(t, X)}^{\frac{1}{2}}\l(\frac{\zeta'}{\zeta}(\s + \i t) - \frac{\zeta'}{\zeta}(\s_{\theta, K}(t, X) + \i t) \r)d\s\\
			&\eqc I_{1} + I_{2} + I_{3},
		\end{align}
		say.
		Using \cref{FLDZ,lem_h_zf,INE_MZ_NZ} and observing the present definition of $u$, we obtain
		\begin{align}
			I_{1} = \sum_{2 \leq n \leq X}\frac{\Lam(n) w_{X}(n)}{n^{s_{\theta, K}(t, X)} \log{n}}
			+ \frac{\Delta_{1}}{\log{X}} \sum_{\rho}\frac{\s_{\theta, K}(t, X) - \b}{(\s_{\theta, K}(t, X) - \b)^{2} + (t - \gamma)^{2}}
			+ O\l( \frac{X^{1/2}}{|t|^{3} (\log{X})^{3}} \r),
		\end{align}
		and
		\begin{align}
			I_{2}
			&= (\s_{\theta, K}(t, X) - \tfrac{1}{2})\l(\sum_{n \leq X}\frac{\Lam(n) w_{X}(n)}{n^{s_{\theta, K}(t, X)}}
			+ \Delta_{2}\sum_{\rho}\frac{\s_{\theta, K}(t, X) - \b}{(\s_{\theta, K}(t, X) - \b)^{2} + (t - \gamma)^{2}}\r)\\
			&\qquad + O\l( (\s_{\theta, K}(t, X) - \tfrac{1}{2})\frac{X^{1/2}}{|t|^{3} (\log{X})^{2}} \r)
		\end{align}
		for some $\Delta_{i} = \Delta_{i}(t, X, \theta, K) \in \CC$ with $|\Delta_{i}| \leq K^{-2}$ for any large $K$.
		It follows from \cref{BFLDRZ} that
		\begin{align}
			\label{Eq:I3zero}
			I_{3}
			&= \sum_{\rho}\int_{\s_{\theta, K}(t, X)}^{\frac{1}{2}}\l( \frac{1}{\s + \i t - \rho} - \frac{1}{\s_{\theta, K}(t, X) + \i t - \rho} \r)d\s
			+ O(1).
		\end{align}

		Next, we consider the real part of $I_{3}$.
		We find that
		\begin{align}
			&\Re\int_{\s_{\theta, K}(t, X)}^{\frac{1}{2}}\l( \frac{1}{\s + \i t - \rho} - \frac{1}{\s_{\theta, K}(t, X) + \i t - \rho} \r)d\s\\
			&= -\frac{1}{2}\log\l( \frac{(\s_{\theta, K}(t, X) - \b)^2 + (t - \gamma)^{2}}{(\b - \frac{1}{2})^{2} + (t - \gamma)^{2}} \r)
			+ \frac{(\s_{\theta, K}(t, X)- \frac{1}{2})(\s_{\theta, K}(t, X) - \b)}{(\s_{\theta, K}(t, X) - \b)^{2} + (t - \gamma)^{2}}.
		\end{align}
		Therefore, we have
		\begin{align}
			\label{Eq_Y_tiY}
			\Re I_{3}
			&= -Y_{\theta, K}(t, X) + \tilde{Y}_{\theta, K}(t, X)\\
			&\quad+ (\s_{\theta, K}(t, X)- \tfrac{1}{2}) \sum_{\rho}\frac{\s_{\theta, K}(t, X) - \b}{(\s_{\theta, K}(t, X) - \b)^{2} + (t - \gamma)^{2}}
			+ O(1),
		\end{align}
		where $Y_{\theta, K}(t, X)$ is defined by \cref{def_Y_hK}, and
		\begin{align}
			\tilde{Y}_{\theta, K}(t, X)
			= \frac{1}{2}\sum_{\rho \not\in S}\log\l( \frac{(\b - \frac{1}{2})^{2} + (t - \gamma)^{2}}{(\s_{\theta, K}(t, X) - \b)^2 + (t - \gamma)^{2}} \r).
		\end{align}
		If $h \leq \frac{1}{\sqrt{2}}$ holds,
		$\tilde{Y}_{\theta, K}(t, X)$ is an empty sum
		since $|t - \gamma| > K X^{|\b - \frac{1}{2}|} / \log{X} > K |\b - \frac{1}{2}| \geq |\b - \frac{1}{2}|$ holds
		when $|\b - \frac{1}{2}| > \frac{1}{\sqrt{2}}(\s_{\theta, K}(t, X) - \frac{1}{2}) \geq h (\s_{\theta, K}(t, X) - \frac{1}{2})$.
		Hence, $\tilde{Y}_{\theta, K}(t, X)$ is zero identically when $h \leq \frac{1}{\sqrt{2}}$.

		We evaluate $\tilde{Y}_{\theta, K}(t, X)$ under the assumption $\frac{1}{\sqrt{2}} < h < 1$.
		By the symmetry of nontrivial zeros, $\tilde{Y}_{\theta, K}(t, X)$ is equal to
		\begin{align}
			&\frac{1}{4}\sum_{\rho \not\in S}\l\{\log\l( \frac{(\b - \frac{1}{2})^{2} + (t - \gamma)^{2}}{(\s_{\theta, K}(t, X) - \b)^2 + (t - \gamma)^{2}} \r)
			+ \log\l( \frac{((1 - \b) - \frac{1}{2})^{2} + (t - \gamma)^{2}}{(\s_{\theta, K}(t, X) - (1 - \b))^2 + (t - \gamma)^{2}} \r)\r\}\\
			&= \frac{1}{4}\sum_{\rho \not\in S}
			\log\l( 1 + \frac{(\s_{\theta, K}(t, X) - \frac{1}{2})^{2}
					\l\{ 2(\b - \frac{1}{2})^{2} - (\s_{\theta, K}(t, X) - \frac{1}{2})^{2} - 2(t - \gamma)^{2} \r\}}{
					\{ (\s_{\theta, K}(t, X) - \b)^2 + (t - \gamma)^{2} \}\{ (\s_{\theta, K}(t, X) - (1 - \b))^2 + (t - \gamma)^{2} \}} \r).
		\end{align}
		We see from the definition of $S$ that the right hand side is nonnegative.
		The inequality $|\b - \frac{1}{2}| > h (\s_{\theta, K}(t, X) - \frac{1}{2})$ yields the inequality $|t - \gamma| > |\b - \frac{1}{2}|$,
		so that also $\rho \in S$.
		Therefore, when $\rho \not\in S$, we have
		\begin{align}
			\label{Eq:pAFLZ3}
			\tfrac{1}{\sqrt{2}}(\s_{\theta, K}(t, X) - \tfrac{1}{2})
			\leq
			\sqrt{\tfrac{1}{2}(\s_{\theta, K}(t, X) - \tfrac{1}{2})^{2} + (t - \gamma)^{2}}
			<  |\b - \tfrac{1}{2}| \leq h (\s_{\theta, K}(t, X) - \tfrac{1}{2}).
		\end{align}
		This inequality leads to
		\begin{align}
			&2(\b - \tfrac{1}{2})^{2} - (\s_{\theta, K}(t, X) - \tfrac{1}{2})^{2} - 2(t - \gamma)^{2}\\
			&\leq \frac{2h^{2} - 1}{2(h^{2} + 1)}\l( 2(\s_{\theta, K}(t, X) - \tfrac{1}{2})^{2} + 2(\b - \tfrac{1}{2})^{2} + 2(t - \gamma)^{2} \r).
		\end{align}
		Hence, $\tilde{Y}_{\theta, K}(t, X)$ is
		\begin{multline}
			\leq \frac{1}{4}\sum_{\rho \not\in S}\log\biggl( 1
			+ \frac{2h^{2} - 1}{2(h^{2} + 1)}
			\biggl(\frac{(\s_{\theta, K}(t, X) - \frac{1}{2})^{2}}{(\s_{\theta, K}(t, X) - \b)^2 + (t - \gamma)^{2}}\\
				+ \frac{(\s_{\theta, K}(t, X) - \frac{1}{2})^{2}}{(\s_{\theta, K}(t, X) - (1 - \b))^2 + (t - \gamma)^{2}}\biggr) \biggr).
		\end{multline}
		It follows from \cref{Eq:pAFLZ3} that
		\begin{align}
			&\frac{(\s_{\theta, K}(t, X) - \frac{1}{2})^{2}}{(\s_{\theta, K}(t, X) - \b)^2 + (t - \gamma)^{2}}
			+ \frac{(\s_{\theta, K}(t, X) - \frac{1}{2})^{2}}{(\s_{\theta, K}(t, X) - (1 - \b))^2 + (t - \gamma)^{2}}\\
			&\geq \frac{(\s_{\theta, K}(t, X) - \frac{1}{2})^{2}}{
				(\s_{\theta, K}(t, X) - \b)^2 + (\b - \frac{1}{2})^{2} - \frac{1}{2}(\s_{\theta, K}(t, X) - \frac{1}{2})^{2}}\\
			&\qquadf\qquadt+ \frac{(\s_{\theta, K}(t, X) - \frac{1}{2})^{2}}{
				(\s_{\theta, K}(t, X) - (1 - \b))^2 + (\b - \frac{1}{2})^{2} - \frac{1}{2}(\s_{\theta, K}(t, X) - \frac{1}{2})^{2}}\\
			&= 4\frac{4(\s_{\theta, K}(t, X) - \frac{1}{2})^{2}(\b - \frac{1}{2})^{2} + (\s_{\theta, K}(t, X) - \frac{1}{2})^{4}}{
			\{4(\b - \frac{1}{2})^{2} - (\s_{\theta, K}(t, X) - \frac{1}{2})^{2}\}^{2}}.
		\end{align}
		The minimum of the function $f(x) = \frac{4a x + a^{2}}{(4x - a)^{2}}$
		in $\frac{1}{2}a \leq x \leq h^{2} a$ for every $a > 0$ is $f(h^{2} a) = \frac{4h^{2} + 1}{(4h^{2} - 1)^{2}}$,
		so we have
		\begin{align}
			4\frac{(\s_{\theta, K}(t, X) - \frac{1}{2})^{4} + 4(\s_{\theta, K}(t, X) - \frac{1}{2})^{2}(\b - \frac{1}{2})^{2}}{
			\{4(\b - \frac{1}{2})^{2} - (\s_{\theta, K}(t, X) - \frac{1}{2})^{2}\}^{2}}
			\geq 4\frac{4h^{2} + 1}{(4h^{2} - 1)^{2}}.
		\end{align}
		By these inequalities and the inequality $\log(1 + x) \leq \frac{\log(1 + x_{0})}{x_{0}}x$ for $x \geq x_{0} \geq 0$,
		$\tilde{Y}_{\theta, K}(t, X)$ is
		\begin{multline}
			\leq \frac{1}{16}
			\frac{(4h^{2} - 1)^{2}}{4h^{2} + 1}\log\l( 1 + 2\frac{(2h^{2} - 1)(4h^{2} + 1)}{(h^{2} + 1)(4h^{2} - 1)^{2}} \r)\\
			\times\sum_{\rho \not\in S}\l(\frac{(\s_{\theta, K}(t, X) - \frac{1}{2})^{2}}{(\s_{\theta, K}(t, X) - \b)^2 + (t - \gamma)^{2}}
			+ \frac{(\s_{\theta, K}(t, X) - \frac{1}{2})^{2}}{(\s_{\theta, K}(t, X) - (1 - \b))^2 + (t - \gamma)^{2}}\r).
		\end{multline}
		Moreover, using the symmetry of nontrivial zeros and \cref{lem_h_zf}, we have
		\begin{multline}
			\tilde{Y}_{\theta, K}(t, X)
			\leq \frac{(4h^{2} - 1)^{2}(1 + h^{2})}{8(4h^{2} + 1)(1 - h^{2})}\log\l( 1 + 2\frac{(2h^{2} - 1)(4h^{2} + 1)}{(h^{2} + 1)(4h^{2} - 1)^{2}} \r)\\
			\times(\s_{\theta, K}(t, X) - \tfrac{1}{2})
			\sum_{\rho}\frac{\s_{\theta, K}(t, X) - \b}{(\s_{\theta, K}(t, X) - \b)^2 + (t - \gamma)^{2}}
		\end{multline}
		when $\frac{1}{\sqrt{2}} < h \leq 1$.
		By this inequality and \cref{Eq_Y_tiY}, we obtain
		\begin{align}
			\Re I_{3}
			&= - Y_{\theta, K}(t, X)
			+ \Delta_{3}(\s_{\theta, K}(t, X) - \tfrac{1}{2})\sum_{\rho}\frac{\s_{\theta, K}(t, X) - \b}{(\s_{\theta, K}(t, X) - \b)^2 + (t - \gamma)^{2}}
			+ O(1)
		\end{align}
		for some $\Delta_{3} = \Delta_{3}(t, X, \theta, K) \in \RR$ such that $\Delta_{3} = 1$ if $0 < h \leq \frac{1}{\sqrt{2}}$, and
		\begin{align}
			&1 \leq \Delta_{3}
			\leq 1 + \frac{(4h^{2} - 1)^{2}(1 + h^{2})}{8(4h^{2} + 1)(1 - h^{2})}\log\l( 1 + 2\frac{(2h^{2} - 1)(4h^{2} + 1)}{(h^{2} + 1)(4h^{2} - 1)^{2}} \r)
		\end{align}
		if $\frac{1}{\sqrt{2}} < h < 1$.
		In particular, the inequality $|\Delta_{3}| \leq a(h)$ also holds with $a(h)$ the number defined by \cref{def_a_h}.

		Next, we evaluate the imaginary part of $I_{3}$.
		When $|\b - \frac{1}{2}| > h (\s_{\theta, K}(t, X) - \frac{1}{2})$, we find from the definition of $\s_{\theta, K}(t, X)$ that
		\begin{align}
			|t - \gamma| > K \frac{X^{|\b - 1/2|}}{\log{X}} > K|\b - \tfrac{1}{2}| > \frac{K h}{1 + h}|\s_{\theta, K}(t, X) - \b|,
		\end{align}
		and so
		\begin{align}
			&\bigg| \Im\int_{\s_{\theta, K}(t, X)}^{\frac{1}{2}}\l( \frac{1}{\s + \i t - \rho} - \frac{1}{\s_{\theta, K}(t, X) + \i t - \rho} \r)d\s\bigg|\\
			&= \frac{|t - \gamma|}{(\s_{\theta, K}(t, X) - \b)^{2} + (t - \gamma)^{2}}\bigg|\int_{\frac{1}{2}}^{\s_{\theta, K}(t, X)}
			\frac{(\s_{\theta, K}(t, X) - \s)(\s_{\theta, K}(t, X) + \s - 2\b)}{(\s - \b)^{2} + (t - \gamma)^{2}}d\s\bigg|\\
			&\leq \frac{1}{(\s_{\theta, K}(t, X) - \b)^{2} + (t - \gamma)^{2}}
			\frac{2(\s_{\theta, K}(t, X) - \tfrac{1}{2}) + 2|\b - \tfrac{1}{2}|}{|t - \gamma|}
			\int_{\frac{1}{2}}^{\s_{\theta, K}(t, X)}(\s_{\theta, K}(t, X) - \s)d\s\\
			&\leq \frac{1}{(\s_{\theta, K}(t, X) - \b)^{2} + (t - \gamma)^{2}}\frac{1}{K|\b - \frac{1}{2}|}
			(\s_{\theta, K}(t, X) - \tfrac{1}{2})^{2}(|\b - \tfrac{1}{2}| + |\s_{\theta, K}(t, X) - \tfrac{1}{2}|)\\
			&\leq \frac{1 + h}{h K}(\s_{\theta, K}(t, X) - \tfrac{1}{2})\frac{\s_{\theta, K}(t, X) - \frac{1}{2}}{
				(\s_{\theta, K}(t, X) - \b)^{2} + (t - \gamma)^{2}}\\
			&\leq \frac{1 + h^{2}}{h(1 - h)K}
			(\s_{\theta, K}(t, X) - \tfrac{1}{2})\frac{\s_{\theta, K}(t, X) - \b}{(\s_{\theta, K}(t, X) - \b)^{2} + (t - \gamma)^{2}}
		\end{align}
		by \cref{lem_h_zf}.
		We also see that
		\begin{align}
			&\Im\int_{\s_{\theta, K}(t, X)}^{\frac{1}{2}}\l( \frac{1}{\s + \i t - \rho} - \frac{1}{\s_{\theta, K}(t, X) + \i t - \rho} \r)d\s\\
			&= \frac{t - \gamma}{(\s_{\theta, K}(t, X) - \b)^{2} + (t - \gamma)^{2}}\int_{\frac{1}{2}}^{\s_{\theta, K}(t, X)}
			\frac{(\s_{\theta, K}(t, X) - \s)(\s_{\theta, K}(t, X) + \s - 2\b)}{(\s - \b)^{2} + (t - \gamma)^{2}}d\s,
		\end{align}
		and that
		\begin{align}
			&(t - \gamma)\int_{\frac{1}{2}}^{\s_{\theta, K}(t, X)}
			\frac{(\s_{\theta, K}(t, X) - \s)(\s_{\theta, K}(t, X) + \s - 2\b)}{(\s - \b)^{2} + (t - \gamma)^{2}}d\s\\
			&= \int_{(1/2 - \b) / (t - \gamma)}^{(\s_{\theta, K}(t, X) - \b) / (t - \gamma)}
			\frac{(\s_{\theta, K}(t, X) - u(t - \gamma) - \b)(\s_{\theta, K}(t, X) + u(t - \gamma) - \b)}{1 + u^{2}}du\\
			&= \arctan\l( \frac{\b - 1/2}{t - \gamma} \r)(\s_{\theta, K}(t, X) - \tfrac{1}{2})(\s_{\theta, K}(t, X) - 2\b + \tfrac{1}{2})\\
			&\qquadf \qquadf+ 2(t - \gamma)^{2}\int_{(1/2 - \b) / (t - \gamma)}^{(\s_{\theta, K}(t, X) - \b) / (t - \gamma)}u \arctan(u)du.
		\end{align}
		By \cref{Eq:I3zero}, these formulas, and the symmetry of nontrivial zeros, we have
		\begin{align}
			\label{ESIMI31}
			\Im I_{3}
			&= \sum_{\rho \in U}\biggl\{
			\tfrac{1}{2}(\s_{\theta, K}(t, X) - \tfrac{1}{2})\Xi_{1}(\rho) + (t - \gamma)^{2}\Xi_{2}(\rho) \biggr\}\\
			&\qquadf + \delta (\s_{\theta, K}(t, X) - \tfrac{1}{2})
			\sum_{\rho \notin U}\frac{\s_{\theta, K}(t, X) - \b}{(\s_{\theta, K}(t, X) - \b)^{2} + (t - \gamma)^{2}}
			+ O(1)
		\end{align}
		for some $\delta = \delta(t, X, \theta, K) \in \RR$ with $|\delta| \leq \frac{1 + h^{2}}{h(1 - h)K}$.
		Here, $U$ is the set of nontrivial zeros with $|\b - \frac{1}{2}| \leq h (\s_{\theta, K}(t, X) - \frac{1}{2})$,
		and $\Xi_{1}(\rho) = \Xi_{1}(\rho; t, X, \theta, K)$ and $\Xi_{2}(\rho) = \Xi_{2}(\rho; t, X, \theta, K)$ are defined by
		\begin{align}
			&\Xi_{1}(\rho)\\
			&\ceq \frac{(\s_{\theta, K}(t, X) - 2\b + \frac{1}{2})\arctan(\frac{\b - 1/2}{t - \gamma})}{(\s_{\theta, K}(t, X) - \b)^{2} + (t - \gamma)^{2}}
			+ \frac{(\s_{\theta, K}(t, X) - 2(1 - \b) + \frac{1}{2})\arctan(\frac{(1 - \b) - 1/2}{t - \gamma})}{
				(\s_{\theta, K}(t, X) - (1 - \b))^{2} + (t - \gamma)^{2}}\\
			\label{Rep_Xi1}
			&= \frac{-4(\b - \frac{1}{2})((\b - \frac{1}{2})^{2} + (t - \gamma)^{2})\arctan(\frac{\b - 1/2}{t - \gamma})}{
				\{ (\s_{\theta, K}(t, X) - \b)^{2} + (t - \gamma)^{2} \}\{ (\s_{\theta, K}(t, X) - (1 - \b))^{2} + (t - \gamma)^{2} \}},
		\end{align}
		and
		\begin{align}
			\Xi_{2}(\rho)
			\ceq \frac{\int_{(1/2 - \b) / (t - \gamma)}^{(\s_{\theta, K}(t, X) - \b) / (t - \gamma)}u \arctan(u)du}{(\s_{\theta, K}(t, X) - \b)^{2} + (t - \gamma)^{2}}
			+ \frac{\int_{(\b - 1/2) / (t - \gamma)}^{(\s_{\theta, K}(t, X) - (1 - \b)) / (t - \gamma)}u \arctan(u)du}{(\s_{\theta, K}(t, X) - (1 - \b))^{2} + (t - \gamma)^{2}}.
		\end{align}
		We see for $\rho \in U$ that $(t - \gamma)\Xi_{1}(\rho) \leq 0$ by \cref{Rep_Xi1}, and that $(t - \gamma)\Xi_{2}(\rho) \geq 0$.
		Therefore, $\Xi_{1}(\rho)$ and $\Xi_{2}(\rho)$ are of distinct signs.
		From this observation, we use the inequality $|a + b| \leq \max\{ |a|, |b| \}$ for $a$ and $b$ being of distinct signs to obtain
		\begin{align}
			&\bigg|\sum_{\rho \in U}\biggl\{\tfrac{1}{2}(\s_{\theta, K}(t, X) - \tfrac{1}{2})\Xi_{1}(\rho) + (t - \gamma)^{2}\Xi_{2}(\rho) \biggr\}\bigg|\\
			\label{ESIMI32}
			&\leq \max\biggl\{\sum_{\rho \in U}\tfrac{1}{2}(\s_{\theta, K}(t, X) - \tfrac{1}{2})|\Xi_{1}(\rho)|,
			\sum_{\rho \in U}(t - \gamma)^{2}|\Xi_{2}(\rho)| \biggr\}.
		\end{align}
		Simple calculations show for $\rho \in U$ that
		\begin{align}
			|\Xi_{1}(\rho)|
			&\leq \frac{\pi}{2}\frac{2|\b - \frac{1}{2}|(2(\s_{\theta, K}(t, X) - \frac{1}{2})^{2} + 2(\b - \frac{1}{2})^{2} + 2(t - \gamma)^{2})}{
				\{ (\s_{\theta, K}(t, X) - \b)^{2} + (t - \gamma)^{2} \}\{ (\s_{\theta, K}(t, X) - (1 - \b))^{2} + (t - \gamma)^{2} \}}\\
			&= \pi\l( \frac{|\b - \frac{1}{2}|}{(\s_{\theta, K}(t, X) - \b)^{2} + (t - \gamma)^{2}}
			+ \frac{|\b - \frac{1}{2}|}{(\s_{\theta, K}(t, X) - (1 - \b))^{2} + (t - \gamma)^{2}} \r)\\
			&\leq \pi h \l( \frac{\s_{\theta, K}(t, X) - \frac{1}{2}}{(\s_{\theta, K}(t, X) - \b)^{2} + (t - \gamma)^{2}}
			+ \frac{\s_{\theta, K}(t, X) - \frac{1}{2}}{(\s_{\theta, K}(t, X) - (1 - \b))^{2} + (t - \gamma)^{2}} \r),
		\end{align}
		and that
		\begin{align}
			&(t - \gamma)^{2}|\Xi_{2}(\rho)|\\
			&\leq\frac{\pi}{4}\l( \frac{(\s_{\theta, K}(t, X) - \b)^{2} \pm (\b - \frac{1}{2})^{2}}{(\s_{\theta, K}(t, X) - \b)^{2} + (t - \gamma)^{2}}
			+ \frac{(\s_{\theta, K}(t, X) - (1 - \b))^{2} \mp (\b - \frac{1}{2})^{2}}{(\s_{\theta, K}(t, X) - (1 - \b))^{2} + (t - \gamma)^{2}} \r)\\
			&\leq \frac{\pi}{2}(1 + h^{2})
			\frac{(\s_{\theta, K}(t, X) - \tfrac{1}{2})^{2}\{(\s_{\theta, K}(t, X) - \frac{1}{2})^{2} + (\b - \frac{1}{2})^{2} + (t - \gamma)^{2}\}}{
				\{ (\s_{\theta, K}(t, X) - \b)^{2} + (t - \gamma)^{2} \}
				\{ (\s_{\theta, K}(t, X) - (1 - \b))^{2} + (t - \gamma)^{2} \}}\\
			&= \frac{\pi}{4}(1 + h^{2})\l( \frac{(\s_{\theta, K}(t, X) - \tfrac{1}{2})^{2}}{(\s_{\theta, K}(t, X) - \b)^{2} + (t - \gamma)^{2}}
			+ \frac{(\s_{\theta, K}(t, X) - \tfrac{1}{2})^{2}}{(\s_{\theta, K}(t, X) - (1 - \b))^{2} + (t - \gamma)^{2}} \r),
		\end{align}
		where $\pm$ coincides with the sign of $\b - \frac{1}{2}$.
		Therefore, we have
		\begin{align}
			\sum_{\rho \in U}\tfrac{1}{2}(\s_{\theta, K}(t, X) - \tfrac{1}{2})|\Xi_{1}(\rho)|
			\leq \pi h (\s_{\theta, K}(t, X) - \tfrac{1}{2})
			\sum_{\rho}\frac{\s_{\theta, K}(t, X) - \frac{1}{2}}{(\s_{\theta, K}(t, X) - \b)^{2} + (t - \gamma)^{2}},
		\end{align}
		and
		\begin{align}
			&\sum_{\rho \in U}
			(t - \gamma)^{2}|\Xi_{2}(\rho)|
			\leq \frac{\pi}{2}(1 + h^{2})(\s_{\theta, K}(t, X) - \tfrac{1}{2})
			\sum_{\rho}\frac{\s_{\theta, K}(t, X) - \frac{1}{2}}{(\s_{\theta, K}(t, X) - \b)^{2} + (t - \gamma)^{2}}.
		\end{align}
		Applying these two inequalities and \cref{ESIMI32} to \cref{ESIMI31} and using \cref{lem_h_zf}, we obtain
		\begin{align}
			\Im I_{3}
			= \Delta_{4} (\s_{\theta, K}(t, X) - \tfrac{1}{2})
			\sum_{\rho}\frac{\s_{\theta, K}(t, X) - \b}{(\s_{\theta, K}(t, X) - \b)^{2} + (t - \gamma)^{2}} + O(1)
		\end{align}
		for some $\Delta_{4} = \Delta_{4}(t, X, \theta, K) \in \RR$ with
		$|\Delta_{4}| \leq \frac{\pi(1 + h^{2})^{2}}{2(1 - h^{2})} + \frac{1 + h^{2}}{h(1 - h)K}
			= b(h) + \frac{1 + h^{2}}{h(1 - h)K}$, where $b(h)$ is the constant defined by \cref{def_b_h}.

		From the above formulas, we have
		\begin{align}
			&\log{\zeta(\tfrac{1}{2} + \i t)}
			= \sum_{2 \leq n \leq X}\frac{\Lam(n) w_{X}(n)}{n^{s_{\theta, K}(t, X)}}\l( \frac{1}{\log{n}} + (\s_{\theta, K}(t, X) - \tfrac{1}{2}) \r)
			- Y_{\theta, K}(t, X)\\
			&+ \l( \frac{\Delta_{1}}{\log{X}} + (\Delta_{2} + \Delta_{3} + \i \Delta_{4})(\s_{\theta, K}(t, X) - \tfrac{1}{2}) \r)
			\sum_{\rho} \frac{\s_{\theta, K}(t, X) - \b}{(\s_{\theta, K}(t, X) - \b)^{2} + (t - \gamma)^{2}} + O(1).
		\end{align}
		Using the prime number theorem, partial summation, and the formula $\sum_{p \leq X}\frac{1}{p} = \log{\log{X}} + O(1)$, we see that
		\begin{align}
			&\sum_{2 \leq n \leq X}\frac{\Lam(n) w_{X}(n)}{n^{s_{\theta, K}(t, X)}}\l( \frac{1}{\log{n}} + (\s_{\theta, K}(t, X) - \tfrac{1}{2}) \r)
			= P_{\theta, K}(t, X)
			+ O((\s_{\theta, K}(t, X) - \tfrac{1}{2}) \log{X}),
		\end{align}
		where $P_{\theta, K}(t, X)$ is defined by \cref{def_PP}.
		Applying \cref{ERZDP} to the term of zeros, we obtain
		\begin{align}
			&\bigg|\Re e^{-\i \theta}\log{\zeta(\tfrac{1}{2} + \i t)}
			- \Re e^{-\i \theta}P_{\theta, K}(t, X)
			+ \Re e^{-\i\theta}Y_{\theta, K}(t, X) \bigg|\\
			&\leq \l(\frac{h}{A(h, \theta)} + \frac{9}{K} \r)(\s_{\theta, K}(t, X) - \tfrac{1}{2})
			\l( \frac{1}{2} \log{|t|} - \Re \sum_{p \leq X}\frac{w_{X}(p) \log{p}}{p^{s_{\theta, K}(t, X)}} + C_{1}\log{X} \r)
		\end{align}
		for large $K$ since
		$\displaystyle{\max_{2 / 5 \leq h \leq 3 / 4}}\frac{1 + h^{2}}{h(1 - h)K}
			= \frac{25}{3 K} < \frac{9}{K}$ holds.
		Here, $A(h, \theta)$ is the positive number defined by \cref{def_Atheta}, and $C_{1}$ is an absolute positive constant.

		Finally, we show $Y_{\theta, K}(t, X) \geq 0$.
		By simple calculations and the symmetry of nontrivial zeros, we see that
		\begin{align}
			&Y_{\theta, K}(t, X)\\
			&=\frac{1}{4}\sum_{\rho \in S}\l(\log\l( \frac{(\s_{\theta, K}(t, X) - \b)^2 + (t - \gamma)^{2}}{(\b - \frac{1}{2})^{2} + (t - \gamma)^{2}} \r)
			+ \log\l( \frac{(\s_{\theta, K}(t, X) - (1 - \b))^2 + (t - \gamma)^{2}}{(\b - \frac{1}{2})^{2} + (t - \gamma)^{2}} \r)\r)\\
			&= \frac{1}{4}\sum_{\rho \in S}\log\l( 1 + \frac{(\s_{\theta, K}(t, X) - \frac{1}{2})^{2}
			\l\{ (\s_{\theta, K}(t, X) - \frac{1}{2})^{2} + 2(t - \gamma)^{2} - 2(\b - \frac{1}{2})^{2} \r\}}{
			\{ (\b - \frac{1}{2})^{2} + (t - \gamma)^{2} \}^{2}} \r)
			\geq 0.
		\end{align}
		This concludes the proof of \cref{AFLZ}.
		\qed

\section{\textbf{Upper bounds for the distribution of $\Re e^{-\i \theta}\log{\zeta(\frac{1}{2} + \i t)}$}}	\label{Sec:UBDRZ}

	We denote the Lebesgue measure in $\RR$ by $\meas$.
	Let $\mathcal{L}$ be a Lebesgue measurable set with a subset of $[T, 2T]$.
	Define
	\begin{align}
		\PP_{T}^{\mathcal{L}}(f \in A)
		= \frac{1}{T}\meas\set{t \in \mathcal{L}}{f(t) \in A}
	\end{align}
	for any $T \geq 1$, any Lebesgue measurable function $f$, and any Lebesgue measurable set $A$ in $\RR$.
	If $\mathcal{L} = [T, 2T]$, we write $\PP_{T}^{\mathcal{L}} = \PP_{T}$.
	Let $T_{1}(K)$ be a sufficiently large number depending only on $K$.
	Under these notations, the goal of this section is to prove the following proposition.

	\begin{proposition}	\label{UBDRZ}
		Let $\lam$ be a positive number, and let $\Phi$ be a nonnegative-valued function satisfying \cref{AZDE}.
		Let $\theta \in [-\frac{\pi}{2}, \frac{\pi}{2}]$, and let $K$ be large.
		For any $T \geq T_{1}(K)$ and any $V \geq (\log{\log{\log{T}}})^{3}$, we have
		\begin{align}
			&\PP_{T}(\Re e^{-\i \theta} \log{\zeta(\tfrac{1}{2} + \i t)} > V)\\
			&\ll \exp\l( -\frac{V^{2}}{4 e K^{2}\log \log T} \r) + \exp\l( - \frac{A(h, \theta)}{K} V \log{V} \r)\\
			&\qquadf\qquadf+ e^{K} \frac{\Phi(T) V}{\log{T}}\exp\l( - \l( 1 - \frac{1}{K} \r)2 A(h, \theta) \lam V \r).
		\end{align}
	\end{proposition}

	\subsection{Preliminaries}
		Define the set $\mathcal{E}_{K}(X, T)$ by
		\begin{align}
			\mathcal{E}_{K}(X, T)
			\ceq \set{t \in [T, 2T]}{\zeta(\s + \i y) \not= 0 \text{ for } \s > \frac{1}{2} + \frac{K}{\log{X}}, |t - y| \leq K \frac{X^{\s - 1 / 2}}{\log{X}}}.
		\end{align}
		It then follows from the definition of $\s_{\theta, K}(t, X)$ that
		\begin{align}	\label{RDFE}
			\mathcal{E}_{K}(X, T)
			= \set{t \in [T, 2T]}{\s_{\theta, K}(t, X) = \frac{1}{2} + \frac{K}{h \log{X}}}.
		\end{align}
		First, we show that the measure of $\mathcal{E}_{K}(X, T)$ is close to $T$.

		\begin{lemma}	\label{EXDZ}
			Let $\lam$ be a positive number, and let $\Phi$ be a nonnegative-valued function satisfying \cref{AZDE}.
			Let $K \geq 1$.
			For any $T \geq T_{1}(K)$ and any $e^{2K} \leq X \leq T^{1 / 2}$, we have
			\begin{align}
				\frac{1}{T}\meas ([T, 2T] \setminus \mathcal{E}_{K}(X, T))
				\ll K e^{K}\frac{\Phi(T)}{\log{X}}\exp\l( -\lam K \frac{\log{T}}{\log{X}} \r).
			\end{align}
		\end{lemma}

		\begin{proof}
			We see that
			\begin{align}
				&\meas ([T, 2T] \setminus \mathcal{E}_{K}(X, T))
				\leq \sum_{\b > \frac{1}{2} + \frac{K}{\log{X}}, T - KX \leq \gamma \leq 2T + KX}2K\frac{X^{\b - 1 / 2}}{\log{X}}\\
				&= 2K \sum_{\b > \frac{1}{2} + \frac{K}{\log{X}}, 0 \leq \gamma \leq 3T}\int_{-\infty}^{\b} X^{\b - 1 / 2}d\s
				= 2K \int_{-\infty}^{1}X^{\s - 1 / 2}\sum_{\substack{\b > \max\{\frac{1}{2} + \frac{K}{\log{X}}, \s\},\\ 0 \leq \gamma \leq 3T}}1 d\s.
			\end{align}
			Therefore, it follows from \cref{AZDE} that
			\begin{align}
				&\frac{1}{T}\meas ([T, 2T] \setminus \mathcal{E}_{K}(X, T))\\
				&\ll K \Phi(T)\l(
				T^{- \lam K / \log{X}}\int_{-\infty}^{\frac{1}{2} + \frac{K}{\log{X}}} X^{\s - 1 / 2}d\s
				+ \int_{\frac{1}{2} + \frac{K}{\log{X}}}^{1}\l( \frac{T^{\lam}}{X} \r)^{-(\s - 1 / 2)} d\s\r)\\
				&\ll K e^{K}\frac{\Phi(T)}{\log{X}}\exp\l( - \lam K \frac{\log{T}}{\log{X}} \r),
			\end{align}
			which is the desired assertion.
		\end{proof}

		\begin{lemma}	\label{SLL}
			Let $T \geq 5$, and let $X \geq 3$. Let $k$ be a positive integer such that $X^{k} \leq T$.
			Then, for any complex numbers $a(p)$, we have
			\begin{align}
				\int_{T}^{2T}\bigg| \sum_{p \leq X}\frac{a(p)}{p^{\i t}} \bigg|^{2k}dt
				\ll T k! \l( \sum_{p \leq X}|a(p)|^2 \r)^{k}.
			\end{align}
			Here, the above sums run over prime numbers.
		\end{lemma}

		\begin{proof}
			This is Lemma 3.1 in \cite{IL2021}.
		\end{proof}

	\subsection{Proof of \cref{UBDRZ}}
		Let $\theta \in [-\frac{\pi}{2}, \frac{\pi}{2}]$, and let $K$ be large.
		Let $T \geq T_{1}(K)$.
		When $t \in \mathcal{E}_{K}(X, T)$ we have $\s_{\theta, K}(t, X) = \frac{1}{2} + \frac{K}{h \log{X}}$,
		so it then follows from \cref{AFLZ} that
		\begin{align}
			&\Re e^{-\i \theta} \log\zeta(\tfrac{1}{2} + \i t)\\
			&\leq \Re e^{-\i \theta} P_{\theta, K}(t, X)
			+ \frac{K + 15}{A(h, \theta)}\frac{1}{\log{X}} \l( \frac{1}{2}\log{T}
			- \Re \sum_{p \leq X}\frac{w_{X}(p) \log{p}}{p^{\frac{1}{2} + \frac{K}{h \log{X}} + \i t}} \r) + O(K)
		\end{align}
		for $3 \leq X \leq T^{6}$ since $\frac{2}{5} < h < \frac{3}{4}$ and $\frac{1}{9} < A(h, \theta) < 1$ hold.
		Here, $P_{\theta, K}(t, X)$ is expressed as
		\begin{align}
			P_{\theta, K}(t, X)
			= \sum_{p \leq X}\l( \frac{w_{X}(p)}{p^{\frac{1}{2} + \frac{K}{h \log{X}} + \i t}}\l( 1 + \frac{K \log{p}}{h \log{X}} \r)
			+ \frac{1}{2p^{1 + \frac{2K}{h \log{X}} + 2 \i t}} \r).
		\end{align}
		Choosing $X = T^{(K + 30) / 2 A(h, \theta) V}$ and using \cref{EXDZ}, we obtain
		\begin{align}
			&\PP_{T}(\Re e^{-\i \theta} \log\zeta(\tfrac{1}{2} + \i t) > V)
			\leq \PP_{T}^{\mathcal{E}_{K}(X, T)}(\Re e^{-\i \theta} \log\zeta(\tfrac{1}{2} + \i t) > V) + \frac{1}{T}\meas\mathcal{E}_{K}(X, T)\\
			\label{PUBDRZ1}
			\leq &\PP_{T}\l(\Re e^{-\i \theta} P_{\theta, K}(t, X)
			- \frac{K + 15}{A(h, \theta)\log{X}}\Re\sum_{p \leq X}\frac{w_{X}(p) \log{p}}{p^{\frac{1}{2}
						+ \frac{K}{h \log{X}} + \i t}} > \frac{V}{K} \r)\\
			&\quad+ O\l( e^{K} \frac{\Phi(T) V}{\log{T}} \exp\l( - \l(1 - \frac{1}{K} \r) 2 A(h, \theta) \lam V \r) \r).
		\end{align}
		When $V > \frac{\log{T}}{\sqrt{\log{\log{T}}}}$, the measure on the right hand side is zero, so we may assume $V \leq \frac{\log{T}}{\sqrt{\log{\log{T}}}}$.
		Then the inequality $X \geq 3$ always holds for any large $T$.
		Put $Y = X^{1 / (\log \log T)^{2}} = T^{(K + 30) / 2 A(h, \theta) V (\log{\log{T}})^{2}}$.
		Using \cref{SLL}, we can show for any positive integer $k$ with $k \leq 2 A(h, \theta) V / (K + 30)$ that
		\begin{align}
			&\int_{T}^{2T}\bigg| \sum_{Y < p \leq X}
			\frac{w_{X}(p)}{p^{\frac{1}{2} + \frac{K}{h \log{X}} + \i t}}\l( 1 + \frac{K \log{p}}{h \log{X}} \r) \bigg|^{2k}dt
			\ll T (3 k \log \log \log T)^{k},
		\end{align}
		\begin{align}
			&\int_{T}^{2T}\bigg| \sum_{p \leq X}\frac{1}{2p^{1 + \frac{2K}{h \log{X}} + 2 \i t}} \bigg|^{2k}dt
			\ll T (C k)^{k},
		\end{align}
		and that
		\begin{align}
			\int_{T}^{2T}\bigg|\frac{K + 15}{A(h, \theta)\log{X}}
			\sum_{p \leq X}\frac{w_{X}(p) \log{p}}{p^{\frac{1}{2} + \frac{K}{h \log{X}} + \i t}} \bigg|^{2k}dt
			\ll T (C k K)^{k}
		\end{align}
		for some absolute constant $C > 0$.
		Therefore, we have
		\begin{align}
			&\PP_{T}\biggl( \bigg| \sum_{Y < p \leq X}
			\frac{w_{X}(p)}{p^{\frac{1}{2} + \frac{K}{h \log{X}} + \i t}}\l( 1 + \frac{K \log{p}}{h \log{X}} \r) \bigg|
			+ \bigg| \sum_{p \leq X}\frac{1}{2p^{1 + \frac{2K}{h \log{X}} + 2 \i t}} \bigg|\\
			&\qquadf \qquadf \qquad + \bigg|
			\frac{K + 15}{A(h, \theta)\log{X}}\sum_{p \leq X}\frac{w_{X}(p) \log{p}}{p^{\frac{1}{2} + \frac{K}{h \log{X}} + \i t}} \bigg|
			> \frac{V}{2K}  \biggr)\\
			&\ll \l( \frac{3 \cdot  k (2K)^{2} \log{\log{{\log{T}}}}}{V^{2}} \r)^{k}
			\leq \exp\l( - \frac{A(h, \theta)}{K}V \log{V} \r)
		\end{align}
		by taking $k = \GS{2 A(h, \theta) V / (K + 30)}$.
		In the last deformation, we used the assumptions $V \geq (\log{\log{\log{T}}})^{3}$ and $T \geq T_{1}(K)$.
		Applying this inequality to \cref{PUBDRZ1}, we have
		\begin{align}
			&\PP_{T}(\Re e^{-\i \theta} \log\zeta(\tfrac{1}{2} + \i t) > V)\\
			&\leq \PP_{T}\l(\Re e^{-\i \theta} \sum_{p \leq Y}
			\frac{w_{X}(p)}{p^{\frac{1}{2} + \frac{K}{h \log{X}} + \i t}}\l( 1 + \frac{K \log{p}}{h \log{X}} \r) > \frac{V}{2K} \r)\\
			&\qquadt + O\l( \exp\l( - \frac{A(h, \theta)}{K} V \log{V} \r) + e^{K} \frac{\Phi(T) V}{\log{T}}\exp\l( - \l( 1 - \frac{1}{K} \r)2 A(h, \theta) \lam V \r) \r).
		\end{align}
		For any $\ell \in \ZZ_{\geq 1}$ with $\ell \leq 2 A(h, \theta) V (\log{\log{T}})^{2} / (K + 30)$, we use \cref{SLL} to obtain
		\begin{align}
			&\int_{T}^{2T}\bigg|\sum_{p \leq Y}
			\frac{w_{X}(p)}{p^{\frac{1}{2} + \frac{K}{h \log{X}} + \i t}}\l( 1 + \frac{K \log{p}}{h \log{X}} \r)\bigg|^{2\ell}dt
			\ll T \ell! (\log{\log{T}})^{\ell}.
		\end{align}
		Hence, we have
		\begin{align}
			\PP_{T}\l(\Re e^{-\i \theta} \sum_{p \leq Y}
			\frac{w_{X}(p)}{p^{\frac{1}{2} + \frac{K}{h \log{X}} + \i t}}\l( 1 + \frac{K \log{p}}{h \log{X}} \r) > \frac{V}{2K} \r)
			\ll ((V / 2K)^{-2} \ell \log{\log{T}})^{\ell}.
		\end{align}
		If $V \leq (\log \log T)^{3/2}$, this is $\leq \exp\l( - V^{2} / 4 e K^{2} \log \log T \r)$ by taking $\ell =\GS{(V / 2K)^{2} / e \log \log T}$,
		and if $V \geq (\log \log T)^{\frac{3}{2}}$, this is $\ll \exp(- V \log{V} (\log\log{{T}})^{\frac{1}{5}})$ by taking $\ell =\GS{V (\log \log T)^{1/4}}$.
		This completes the proof of \cref{UBDRZ}.
		\qed

\section{\textbf{Exponential moments of Dirichlet polynomials}}	\label{Sec:EMDP}
	In this section, we evaluate exponential moments of Dirichlet polynomials by applying Harper's method to our approximate formula.
	In the following, we regard that empty sums are zero, and empty products are one.
	Let $\k = k + 3$.
	Put $L = (\log{T} / \log{\log{T}})^{1/8}$.
	Define the sequence $\{ \delta_{i} \}$ by $\delta_{0} = 0$, and
	\begin{align} \label{def_beta}
		\delta_{i} \ceq \frac{1}{(L + 1 - i)^{8}}
	\end{align}
	for $i \in \ZZ_{\geq 1}$, and define
	\begin{align}
		\I = \I _{k, K, T}
		\ceq 1 + \max\set{i \in \ZZ_{\geq 0}}{\delta_{i} \leq e^{- 80 A(h, \theta) \k K}}.
	\end{align}
	Put
	\begin{align}
		\phi_{j}(p)
		&= \phi_{j}(p; T, \theta, K)\\
		&\ceq \frac{w_{T^{\delta_{j}}}(p)}{p^{K / h \log{T^{\delta_{j}}}}}
		\l(e^{-\i\theta} + \frac{e^{-\i\theta} K \log{p}}{h \log{T^{\delta_{j}}}}
		- \l(\frac{K}{A(h, \theta)} + \frac{9}{h}\r)\frac{\log{p}}{\log{T^{\delta_{j}}}}\r),
	\end{align}
	and
	\begin{align}
		\psi_{j}(p)
		= \psi_{j}(p; T, \theta, K)
		\ceq \frac{e^{-\i\theta}}{p^{2K / h \log{T^{\delta_{j}}}}}.
	\end{align}
	We remark that when $K$ is large, $|\phi_{j}(p)| < 8$ holds in the ranges $\frac{2}{5} < h < \frac{3}{4}$ and $1 < \frac{1}{A(h, \theta)} < 9$.
	For each $1 \leq i \leq j \leq \I $, set
	\begin{align}
		a_{j}(p, t)
		= a_{j}(p, t; T, \theta, K)
		\ceq \Re \l(\frac{\phi_{j}(p)}{p^{\frac{1}{2}}}p^{-\i t} + \frac{\psi_{j}(p)}{2p}p^{-2\i t} \r),
	\end{align}
	and
	\begin{align}   \label{def_Gi}
		G_{(i, j)}(t) = G_{(i, j)}(t; T, \theta, K)
		\ceq \sum_{T^{\delta_{i - 1}} < p \leq T^{\delta_{i}}}a_{j}(p, t).
	\end{align}
	Define the sets $\mathcal{A}(i, j) = \mathcal{A}_{k, T, \theta, K}(i, j)$ and $\mathcal{B}(j) = \mathcal{B}_{T, K}(j)$ by
	\begin{gather}
		\label{def_Aij}
		\mathcal{A}(i, j)
		= \set{t \in [T, 2T]}{| G_{(i, j)}(t) | \leq \frac{1}{20 e^{2} \k (j + 1 - i)^2 \delta_{i}}},\\
		\label{def_Bj}
		\mathcal{B}(j)
		= \set{t \in [T, 2T]}{\s_{\theta, K}(t, T^{\delta_{j}}) = \frac{1}{2} + \frac{K}{h \log{T^{\delta_{j}}}}}.
	\end{gather}
	It then holds that $\mathcal{B}(j) = \mathcal{E}_{K}(T^{\delta_{j}}, T)$ by \cref{RDFE}.
	We define
	$\mathcal{T} = \mathcal{T}_{k, T, \theta, K}$, $\mathcal{S}(j) = \mathcal{S}_{k, T, \theta, K}(j)$ by
	\begin{align} \label{def_calT}
		\mathcal{T}
		&\ceq \bigcap_{1 \leq i \leq \I}\mathcal{A}(i, \I ) \cap \mathcal{B}(\I),\\
		\label{def_calS0}
		\Sc(0)
		&\ceq \bigcup_{1 \leq \ell \leq \I}\mathcal{A}(1, \ell)^{c} \cup \mathcal{B}(1)^{c},
	\end{align}
	and
	\begin{align} \label{def_calSj}
		\mathcal{S}(j)
		&\ceq \bigcap_{1 \leq i \leq j}\mathcal{A}(i, j) \cap \mathcal{B}(j)
		\cap \l(\bigcup_{j+ 1 \leq \ell \leq \I}\mathcal{A}(j + 1, \ell)^{c} \cup \mathcal{B}(j + 1)^{c}\r)
	\end{align}
	for $1 \leq j \leq \I - 1$.
	Here, $\mathcal{A}(i, j)^{c}$, $\mathcal{B}(j)^{c}$ stand for their complements in $[T, 2T]$,
	that is, $\mathcal{A}(i, j)^{c} = [T, 2T] \setminus \mathcal{A}(i, j)$, $\mathcal{B}(j)^{c} = [T, 2T] \setminus \mathcal{B}(j)$.
	It then holds that $[T, 2T] = \mathcal{T} \cup \bigcup_{j = 0}^{\I - 1}\mathcal{S}(j)$, which is proved as follows.
	For any $t \in [T, 2T] \setminus \mathcal{T}$, we can take an $i \in \{1, \dots, \I \}$
	such that $\s_{\theta, K}(t, T^{\delta_{i}}) > \frac{1}{2} + \frac{K}{h\log{T^{\delta_{i}}}}$ or
	$|G_{(i, \ell)}(t)| > \frac{1}{20 e^{2} \k (\ell + 1 - i)^2 \delta_{i}}$ for some $i \leq \ell \leq \I $,
	and then $t \in \mathcal{S}(i_{0} - 1)$ with $i_{0}$ the minimum of such $i$'s.
	We also note that the partition is not always disjoint.

	Under these notations, we prove the following key propositions.

	\begin{proposition}   \label{lem_T}
		Let $\theta \in \RR$, and let $K \geq 1$.
		For any $k \geq 0$, $T \geq e^{(1000 k \log{\k})^{2}} T_{0}$ with $T_{0}$ a large constant, we have
		\begin{align}
			&\int_{\mathcal{T}}\exp\l( 2k \sum_{p \leq T^{\delta_{\I}}}a_{\I}(p, t)\r)dt
			\ll e^{-k^2 (\log \log \k + O(1))} T (\log{T^{\delta_{\I}}})^{k^2}.
		\end{align}
		This implicit constant is absolute.
	\end{proposition}

	\begin{proposition}   \label{lem_Sj}
		Let $\lam$ be a positive number, and let $\Phi$ be a nonnegative-valued function satisfying \cref{AZDE}.
		Let $\theta \in \RR$, and let $K$ be large.
		Let $a, b > 1$ with $\frac{1}{a} + \frac{1}{b} = 1$.
		For any $k \geq 0$, $T \geq e^{(1200k\log{\k})^{2}} T_{1}(K)$, we have
		$
			\meas \Sc(0) \ll T (\Phi(T) + 1) \exp\l( - \frac{\log{T}}{\log{\log{T}}}  \r)
		$, and
		\begin{align}
			\label{lem_Sj_1}
			&\int_{\mathcal{S}(j)}\exp\l( 2k \sum_{p \leq T^{\delta_{j}}}a_{j}(p, t) \r)dt\\
			&\ll e^{- \frac{1}{21} \delta_{j}^{-1} \log(1 / \delta_{j + 1})}
			e^{-k^{2}(\log{\log{\k}} + O(1))} T (\log{T^{\delta_{j}}})^{k^{2}}\\
			&\qquadt + \l( K e^{K} \frac{\Phi(T)}{\log{T}} \r)^{1/a} e^{- (\lam K - 1)\delta_{j + 1}^{-1} / a}
			e^{-k^{2} b (\log{\log{(\k b)}} + O(1))} T (\log{T^{\delta_{j}}})^{k^{2} b}
		\end{align}
		for any $1 \leq j \leq \I - 1$.
		These implicit constants are absolute.
	\end{proposition}

	Although there are some differences from the situation of \cite{H2013},
	\cref{lem_T,lem_Sj} involve a small improvement of Lemma 3 in \cite{H2013} for the dependence of $k$ and $T$.
	Thanks to the improvement, we can prove \cref{ES:UBMKt} in the range $k \leq \log{\log{(T / T_{0})}}$.

	\subsection{Preliminaries}

		Let $\mathcal{P}$ be the set of prime numbers.
		Let $\{\mathfrak{X}(p)\}_{p \in \mathcal{P}}$ be a sequence of independent random variables on a probability space $(\Omega, \mathscr{A}, \PP)$
		which are uniformly distributed on the unit circle in $\CC$.
		For each $1 \leq i \leq j \leq \I $, set
		\begin{align}
			\label{def_raj}
			a_{j}(p, \mathfrak{X})
			= a_{j}(p, \mathfrak{X}; T, \theta, K)
			\ceq \Re\l(\frac{\phi_{j}(p)}{p^{\frac{1}{2}}} \mathfrak{X}(p) + \frac{\psi_{j}(p)}{2 p} \mathfrak{X}(p)^{2} \r),
		\end{align}
		and
		\begin{align}
			\label{def_rGi}
			G_{(i, j)}(\mathfrak{X}) = G_{(i, j)}(\mathfrak{X}; T, \theta, K)
			\ceq \sum_{T^{\delta_{i - 1}} < p \leq T^{\delta_{i}}}a_{j}(p, \mathfrak{X}).
		\end{align}

		\begin{lemma} \label{GRLR}
			Let $\{b_{1}(m)\}, \dots, \{ b_{n}(m) \}$ be complex sequences.
			For any distinct prime numbers $q_{1}, \dots, q_{s}$ and any $k_{1, 1}, \dots k_{1, n}, \dots,  k_{s, 1}, \dots, k_{s, n} \in \ZZ_{\geq 1}$,
			we have
			\begin{align}
				&\int_{T}^{2T}\prod_{\ell = 1}^{s} \prod_{j = 1}^{n}\Re b_{j}(q_{\ell}^{k_{j, \ell}}) q_{\ell}^{-\i t k_{j, \ell}}dt\\
				& =  T\EXP{\prod_{\ell = 1}^{s} \prod_{j = 1}^{n}\Re b_{j}(q_{\ell}^{k_{j, \ell}}) \mathfrak{X}(q_{\ell})^{k_{j, \ell}}}
				+ O\l( \prod_{\ell = 1}^{s}\prod_{j = 1}^{n}q_{\ell}^{k_{j, \ell}}|b_{j}(q_{\ell}^{k_{j, \ell}})| \r).
			\end{align}
			Here, $\EXP{f}$ is the expectation of $f$.
		\end{lemma}

		\begin{proof}
			This is Lemma 4.2 in \cite{IL2021}.
		\end{proof}

		\begin{lemma}	\label{MEDP}
			Let $T$ be large.
			Let $1 \leq j \leq \I - 1$, $j + 1 \leq \ell \leq \I$.
			For any nonnegative integers $n_{i}$ satisfying $n_{i} \leq 5^{-1}(j + 1 - i)^{-2} \delta_{i}^{-1}$ with $1 \leq i \leq j$,
			and any nonnegative integers $n_{j + 1}$ satisfying $n_{j + 1} \leq (2 \delta_{j + 1})^{-1}$, we have
			\begin{align}
				&\int_{T}^{2T} G_{(1, j)}(t)^{n_{1}} \cdots G_{(j, j)}(t)^{n_{j}} G_{(j + 1, \ell)}(t)^{n_{j + 1}} dt\\
				&= T \EXP{G_{(j + 1, \ell)}(\mathfrak{X})^{n_{j + 1}}}\prod_{1 \leq i \leq j}\EXP{G_{(i, j)}(\mathfrak{X})^{n_{i}}} + O\l( T^{1/2} \r).
			\end{align}
			The implicit constant is absolute.
		\end{lemma}

		\begin{proof}
			We write
			\begin{align}
				&\int_{T}^{2T} G_{(1, j)}(t)^{n_{1}} \cdots G_{(j, j)}(t)^{n_{j}} G_{(j + 1, \ell)}(t)^{n_{j + 1}} dt =\\
				&\sum_{T^{\delta_{0}} < p_{1, 1}, \dots, p_{1, n_{1}} \leq T^{\delta_{1}}} \cdots \sum_{T^{\delta_{j}} < p_{j + 1, 1}, \dots, p_{j + 1, n_{j + 1}} \leq T^{\delta_{j + 1}}}
				\int_{T}^{2T} a_{*}(p_{1, 1}, t) \cdots a_{*}(p_{j + 1, n_{j + 1}}, t)dt,
			\end{align}
			where $* = j$ if its argument is $p_{i, k}$ with $1 \leq i \leq j$, and $* = \ell$ otherwise.
			The integral on the right hand side can be represented by
			\begin{align}
				\sum_{\substack{\e_{1, 1, 1}, \dots, \e_{1, j+1, n_{j + 1}} \in \{0, 1\},\\
						\e_{2, 1, 1}, \dots, \e_{2, j+1, n_{j + 1}} \in \{0, 1\}, \\
						\e_{1, i, k} + \e_{2, i, k} = 1}}
				\int_{T}^{2T}\prod_{\substack{1 \leq i \leq j + 1 \\ 1 \leq k \leq n_{i}}}
				\l( \Re \frac{\phi_{*}(p_{i, k})}{p_{i, k}^{1/2}} p_{i, k}^{- \i t} \r)^{\e_{1, i, k}}
				\l( \Re \frac{\psi_{*}(p)}{2 p_{i, k}}p_{i, k}^{- 2 \i t} \r)^{\e_{2, i, k}}dt.
			\end{align}
			Using \cref{GRLR}, we find that this is equal to
			\begin{align}
				&\sum_{\substack{\e_{1, 1, 1}, \dots, \e_{1, j + 1, n_{j + 1}} \in \{0, 1\},\\
						\e_{2, 1, 1}, \dots, \e_{2, j + 1, n_{j + 1}} \in \{0, 1\},
						\\ \e_{1, i, k} + \e_{2, i, k} = 1}}
				T \cdot \EXP{\prod_{\substack{1 \leq i \leq j + 1 \\ 1 \leq k \leq n_{i}}}
					\l( \Re \frac{\phi_{*}(p_{i, k})}{p_{i, k}^{1/2}} \mathfrak{X}(p_{i, k}) \r)^{\e_{1, i, k}}
					\l( \Re \frac{\psi_{*}(p)}{2 p_{i, k}}\mathfrak{X}(p_{i, k})^{2} \r)^{\e_{2, i, k}}}\\
				& + O\l(\sum_{\substack{\e_{1, 1, 1}, \dots, \e_{1, j + 1, n_{j + 1}} \in \{0, 1\},\\
						\e_{2, 1, 1}, \dots, \e_{2, j + 1, n_{j + 1}} \in \{0, 1\},\\
						\e_{1, i, k} + \e_{2, i, k} = 1}}
				\prod_{\substack{1 \leq i \leq j + 1 \\ 1 \leq k \leq n_{i}}}( |\phi_{*}(p_{i, k})|p_{i, k}^{1/2})^{\e_{1, i, k}} \r)\\
				&= T \EXP{a_{*}(p_{1, 1}, \mathfrak{X}) \cdots a_{*}(p_{j + 1, n_{j + 1}}, \mathfrak{X})}
				+O\l( (8 p_{1, 1}^{1/2} + 1) \cdots (8 p_{j + 1, n_{j + 1}}^{1/2} + 1) \r).
			\end{align}
			Remark that we used the inequalities $|\phi_{j}(p)| \leq 8$, $|\psi_{j}(p)| \leq 1$ in the above deformation.
			Therefore, we have
			\begin{align}
				&\int_{T}^{2T} G_{(1, j)}(t)^{n_{1}} \cdots G_{(j, j)}(t)^{n_{j}} G_{(j + 1, \ell)}(t)^{n_{j + 1}} dt\\
				&= \EXP{G_{(1, j)}(\mathfrak{X})^{n_{1}} \cdots G_{(j, j)}(\mathfrak{X})^{n_{j}} G_{(j + 1, \ell)}(\mathfrak{X})^{n_{j + 1}}}
				+ O\l( \prod_{i = 1}^{j + 1}\l(\sum_{T^{\delta_{i - 1}} < p \leq T^{\delta_{i}}}(8 p^{1/2} + 1)\r)^{n_{i}} \r).
			\end{align}
			By the independence of $\{ \mathfrak{X}(p) \}$, the above main term is equal to
			\begin{align}
				T \EXP{G_{(j + 1, \ell)}(\mathfrak{X})^{n_{j + 1}}}\prod_{1 \leq i \leq j}\EXP{G_{(i, j)}(\mathfrak{X})^{n_{i}}}.
			\end{align}
			Moreover, we use the prime number theorem and partial summation to obtain
			\begin{align}
				\prod_{i = 1}^{j + 1}\l(\sum_{T^{\delta_{i - 1}} < p \leq T^{\delta_{i}}}(8 p^{1/2} + 1)\r)^{n_{i}}
				\leq T^{(\delta_{1} n_{1} + \cdots + \delta_{j + 1} n_{j + 1})/2}
				\leq T^{\zeta(2)/10 + 1/4} \leq T^{1/2},
			\end{align}
			which completes the proof of \cref{MEDP}.
		\end{proof}

		\begin{lemma} \label{UBDRP}
			Let $\{a(p)\}_{p \in \mathcal{P}}$ be a complex sequence.
			Then, for any $k \in \ZZ_{\geq 1}$, $\ell \in \ZZ \setminus \{ 0 \}$, $X \geq 3$, we have
			\begin{align}
				\EXP{\bigg| \sum_{p \leq X}a(p)\mathfrak{X}(p)^{\ell} \bigg|^{2k}}
				\leq k! \l( \sum_{p \leq X}|a(p)|^2 \r)^{k}.
			\end{align}
		\end{lemma}

		\begin{proof}
			Since $\mathfrak{X}(p)$'s are independent and uniformly distributed on the unit circle in $\CC$, it holds that
			\begin{align}	\label{BEUIR}
				\EXP{ \frac{\mathfrak{X}(p_1)^{a_{1}} \cdots \mathfrak{X}(p_{k})^{a_{k}}}{ \mathfrak{X}(q_1)^{b_1} \cdots \mathfrak{X}(q_{\ell})^{b_{\ell}} } }
				=
				\begin{cases}
					1 & \text{if $p_{1}^{a_1} \cdots p_{k}^{a_{k}} = q_{1}^{b_1} \cdots q_{\ell}^{b_\ell}$},    \\
					0 & \text{if $p_{1}^{a_1} \cdots p_{k}^{a_{k}} \neq q_{1}^{b_1} \cdots q_{\ell}^{b_\ell}$}.
				\end{cases}
			\end{align}
			Hence, we have
			\begin{align}
				\EXP{ \bigg| \sum_{p \leq X}a(p) \mathfrak{X}(p)^{\ell} \bigg|^{2k} }
				& = \sum_{\substack{p_1, \dots, p_{k} \leq X \\ q_{1}, \dots, q_{k} \leq X}}
				a(p_1) \cdots a(p_{k}) \ol{a(q_{1}) \cdots a(q_{k})}
				\EXP{\frac{\mathfrak{X}^{\ell}(p_1) \cdots \mathfrak{X}^{\ell}(p_{k}) }{ \mathfrak{X}^{\ell}(q_{1}) \cdots \mathfrak{X}^{\ell}(q_{k})} } \\
				& \leq k!\sum_{p_{1}, \dots, p_{k} \leq X}|a(p_{1})|^2 \cdots |a(p_{k})|^{2}
				\leq k! \l( \sum_{p \leq X}|a(p)|^2 \r)^{k},
			\end{align}
			which completes the proof of the lemma.
		\end{proof}

	\subsection{Proofs of \cref{lem_T,lem_Sj}}

		\begin{proof}[Proof of \cref{lem_T}]
			Denote the integral we evaluate by $I$.
			Put $B_{i} = 10^{-1}(\I + 1 - i)^{-2} \delta_{i}^{-1}$.
			Using Taylor's theorem, we have $e^{x} = \l( 1 - \frac{x^{N + 1}}{(N + 1)!}e^{\xi x} \r)^{-1}\sum_{n = 1}^{N}\frac{x^{n}}{n!}$
			for some $\xi \in (0, 1)$.
			By this formula, it holds for $t \in \mathcal{T}$ that
			\begin{align}
				\label{INE:EXTR}
				& \exp\l( 2k \Re \sum_{p \leq T^{\delta_{\I}}}a_{\I}(p, t) \r)
				= \prod_{1 \leq i \leq \I}\exp\l( 2k G_{(i, \I )}(t) \r)\\
				&\leq \prod_{1 \leq i \leq \I}\l(\l( 1 - e^{- B_{i}/2} \r)^{-1}\sum_{0 \leq n \leq B_{i}}\frac{(k G_{(i, \I )}(t))^{n}}{n!}\r)^{2}
				\leq C \prod_{1 \leq i \leq \I}\l(\sum_{0 \leq n \leq B_{i}}\frac{(k G_{(i, \I )}(t))^{n}}{n!}\r)^{2}
			\end{align}
			with $C > 0$ an absolute constant.
			Therefore, we have
			\begin{align}   \label{BD_I1}
				I
				\ll \int_{T}^{2T}\prod_{1 \leq i \leq \I} \l(\sum_{0 \leq n \leq B_{i}}\frac{(k G_{(i, \I )}(t))^{n}}{n!}\r)^{2}dt
				= \int_{T}^{2T}\prod_{1 \leq i \leq \I} \sum_{0 \leq m \leq B_{i}}\sum_{0 \leq n \leq B_{i}}
				\frac{(k G_{(i, \I )}(t))^{m + n}}{m! n!}dt.
			\end{align}
			The integral on the right hand side is
			\begin{align}
				& = \int_{T}^{2T} \sum_{0 \leq m_{1} \leq B_{1}} \sum_{0 \leq n_{1} \leq B_{1}}
				\cdots \sum_{0 \leq m_{\I} \leq B_{\I}} \sum_{0 \leq n_{\I} \leq B_{\I}}
				\frac{(k G_{(1, \I )}(t))^{m_{1} + n_{1}} \cdots (k G_{(\I, \I )}(t))^{m_{\I} + n_{\I}}}{m_{1}! n_{1}! \cdots m_{\I}! n_{\I}!}dt\\
				& = \sum_{0 \leq m_{1} \leq B_{1}} \cdots \sum_{0 \leq n_{\I} \leq B_{\I}}
				\frac{k^{m_{1} + n_{1} + \cdots + m_{\I} + n_{\I}}}{m_{1}! n_{1}! \cdots m_{\I} n_{\I}!}
				\int_{T}^{2T}G_{(1, \I )}(t)^{m_{1} + n_{1}} \cdots G_{(\I , \I )}(t)^{m_{\I} + n_{\I}} dt.
			\end{align}
			Using \cref{MEDP}, we find that
			\begin{align}
				& \int_{T}^{2T} G_{(1, \I )}(t)^{m_{1} + n_{1}} \cdots G_{(\I , \I )}(t)^{m_{\I} + n_{\I}} dt
				= T \prod_{1 \leq i \leq \I}\EXP{G_{(i, \I )}(\mathfrak{X})^{m_{i} + n_{i}}} + O\l( T^{1/2} \r).
			\end{align}
			Hence, we have
			\begin{align}
				\label{BdI2}
				I
				&\ll T \prod_{1 \leq i \leq \I}\sum_{0 \leq m \leq B_{i}} \sum_{0 \leq n \leq B_{i}}
				\frac{k^{m + n}}{m! n!}\EXP{G_{(i, \I )}(\mathfrak{X})^{m + n}}
				+ e^{2k \I}T^{1/2}.
			\end{align}
			The second term on the right hand side is $\leq T$
			since $\I \leq L = (\log{T})^{1/8}$ and $T \geq e^{(1000k\log{\k})^{2}}T_{0} \geq e^{k^{2}}$.

			Write
			\begin{align}
				\label{SMEI}
				\prod_{1 \leq i \leq \I}\sum_{0 \leq m \leq B_{i}} \sum_{0 \leq n \leq B_{i}}
				\frac{k^{m + n}}{m! n!}\EXP{G_{(i, \I )}(\mathfrak{X})^{m + n}}
				= M - \sum_{1 \leq \ell \leq \I}E(\ell),
			\end{align}
			where
			\begin{align}
				M = \prod_{1 \leq i \leq \I}\sum_{m = 0}^{\infty} \sum_{n = 0}^{\infty}\frac{k^{m + n}}{m! n!}\EXP{G_{(i, \I)}(\mathfrak{X})^{m + n}}
				= \EXP{\exp\l( 2k\sum_{1 \leq i \leq \I}G_{(i, \I)}(\mathfrak{X}) \r)},
			\end{align}
			and
			\begin{align}
				E(\ell)
				&= \l(\prod_{1 \leq i \leq \ell - 1}\sum_{m = 0}^{\infty}\sum_{n = 0}^{\infty}
				\frac{k^{m + n}}{m! n!}\EXP{G_{(i, \I)}(\mathfrak{X})^{m + n}}\r)\\
				&\qquad\times \l(2\sum_{m = 0}^{\infty} - \sum_{m > B_{\ell}}\r)\sum_{n > B_{\ell}}
					\frac{k^{m + n}}{m! n!}\EXP{G_{(\ell, \I)}(\mathfrak{X})^{m + n}}\\
				&\qquad \times \l(\prod_{\ell + 1 \leq i \leq \I}\sum_{0 \leq m \leq B_{i}}\sum_{0 \leq n \leq B_{i}}
					\frac{k^{m + n}}{m! n!}\EXP{G_{(i, \I)}(\mathfrak{X})^{m + n}}\r).
			\end{align}
			We see, by the independence of $\mathfrak{X}(p)$, that
			\begin{align}
				M
				= \prod_{p \leq T^{\delta_{\I}}}\EXP{\exp\l(2k a_{\I}(p, \mathfrak{X}) \r)}.
			\end{align}
			It follows from $|a_{\I}(p, \mathfrak{X})| \leq \frac{|\phi_{j}(p)|}{p^{1/2}} + \frac{|\psi_{j}(p)|}{2p} \leq \frac{9}{p^{1/2}}$ that
			\begin{align}
				\prod_{p \leq (18 k)^{2}}\EXP{\exp\l(2k a_{\I}(p, \mathfrak{X}) \r)}
				\leq e^{O(k^2)}.
			\end{align}
			Moreover, we use the formulas $\EXP{a_{\I}(p, \mathfrak{X})} = 0$,
			\begin{align}
				\EXP{a_{\I}(p, \mathfrak{X})^{2}}
				= \dfrac{w_{T^{\delta_{\I}}}(p)^{2}}{2 p^{1 + 2K / h \log{T^{\delta_{\I}}}}}
				+ O\l( \frac{1}{p^{1 + K / h \log{T^{\delta_{\I}}}}} \frac{K \log{p}}{\log{T^{\delta_{\I}}}} \r),
			\end{align}
			\begin{align}
				\sum_{p \leq X}\frac{1}{p^{2\s}}
				= \log\l( \min\l\{ \frac{1}{\s - \frac{1}{2}}, \log{X} \r\} \r) + O(1), \qquad
				\sum_{p \leq X}\frac{\log{p}}{p^{2\s}}
				\ll \frac{1}{\s - \frac{1}{2}}
			\end{align}
			for $\s > \frac{1}{2}$ to obtain
			\begin{align}
				&\prod_{(18k)^{2} < p \leq T^{\delta_{\I}}}\EXP{\exp\l(2k a_{\I}(p, \mathfrak{X}) \r)}\\
				&= \prod_{(18k)^{2} < p \leq T^{\delta_{\I}}}\EXP{1 + 2k a_{\I}(p, \mathfrak{X}) + 2 k^{2} a_{\I}(p, \mathfrak{X})^{2}
				+ O\l( \frac{k^{3}}{p^{3/2}} \r)}\\
				&= \prod_{(18k)^{2} < p \leq T^{\delta_{\I}}}\l( 1 + \frac{k^{2} w_{T^{\delta_{\I}}}(p)^{2}}{p^{1 + 2K / h \log{T^{\delta_{\I}}}}}
				+ O\l( \frac{k^{2}}{\log{T^{\delta_{\I}}}}\frac{K \log{p}}{p^{1 + K / h \log{T^{\delta_{\I}}}}} + \frac{k^{3}}{p^{3/2}} \r) \r)\\
				&= \exp\l( k^{2}\sum_{(18k)^{2} < p \leq T^{\delta_{\I} / 3}}\frac{1}{p^{{1 + 2K / h \log{T^{\delta_{\I}}}}}} + O(k^{2}) \r)
				= e^{-k^{2}(\log{\log{\k}} + O(1))} (\log{T^{\delta_{\I}}})^{k^{2}}.
			\end{align}
			Hence, we have
			\begin{align}
				\label{EMI}
				M = e^{-k^{2}(\log{\log{\k}} + O(1))} (\log{T^{\delta_{\I}}})^{k^{2}}.
			\end{align}

			Next, we consider $E(m)$.
			Our goal is to show that
			\begin{align}
				\label{EETI}
				\bigg|\sum_{1 \leq \ell \leq \I}E(\ell)\bigg|
				\ll e^{-k^2(\log{\log{\k}} + O(1))}(\log{T^{\delta_{\I}}})^{k^{2}}.
			\end{align}
			By the Cauchy-Schwarz inequality, \cref{UBDRP}, the inequalities $|\phi_{j}(p)| \leq 8$, $|\psi_{j}(p)| \leq 1$,
			and routine calculations, we find that
			\begin{align}
				\EXP{\bigg|\sum_{T^{\delta_{i - 1}} < p \leq T^{\delta_{i}}} \frac{\phi_{\I}(p)}{p^{1/2}} \mathfrak{X}(p)\bigg|^{n}}
				&\leq (n!)^{1/2} \l( \sum_{T^{\delta_{i - 1}} < p \leq T^{\delta_{i}}}\frac{|\phi_{\I}(p)|^{2}}{p} \r)^{n/2}\\
				&\leq \l\{
				\begin{array}{ll}
					(n!)^{1/2}(9^{2} \log{\log{\log{T}}})^{n/2} 						& \text{if \ $i = 1$,}\\
					(n!)^{1/2}\l(9^{3} / (L - i)\r)^{n/2}	& \text{if \ $2 \leq i \leq \I$,}
				\end{array}
				\r.
			\end{align}
			and that for any $1 \leq i \leq \I$,
			\begin{align}
				\EXP{\bigg|\sum_{T^{\delta_{i - 1}} < p \leq T^{\delta_{i}}} \frac{\psi_{\I}(p)}{2p} \mathfrak{X}(p)^{2}\bigg|^{n}}
				\leq (n!)^{1/2} \l( \sum_{p > T^{\delta_{i - 1}}}\frac{1}{4p^{2}} \r)^{\frac{n}{2}}
				\leq (n!)^{1/2} \l( \frac{1}{T^{\delta_{i - 1}}} \r)^{\frac{n}{2}}.
			\end{align}
			Therefore, we find that
			\begin{align}
				&\EXP{|G_{(i, \I)}(\mathfrak{X})|^{n}}\\
				&\leq 2^{n} \EXP{\bigg|\sum_{T^{\delta_{i - 1}} < p \leq T^{\delta_{i}}} \frac{\phi_{\I}(p)}{p^{1/2}} \mathfrak{X}(p)\bigg|^{n}}
				+ 2^{n}\EXP{\bigg|\sum_{T^{\delta_{i - 1}} < p \leq T^{\delta_{i}}} \frac{\psi_{\I}(p)}{2p} \mathfrak{X}(p)^{2}\bigg|^{n}}\\
				\label{nMG1}
				&\leq \l\{
				\begin{array}{cl}
					20^{n} (n!)^{1/2}(\log{\log{\log{T}}})^{n / 2}		& \text{if \; $i = 1$,}\\
					(n!)^{1/2}\l( \frac{60}{(L - i + 1)^{1/2}} \r)^{n}	& \text{if \; $2 \leq i \leq \I$.}
				\end{array}
				\r.
			\end{align}

			If $\I = 1$, we see that
			\begin{align}
				&\bigg|\sum_{1 \leq \ell \leq \I}E(\ell)\bigg|
				= \l|\l( 2\sum_{m = 0}^{\infty} - \sum_{m > B_{1}} \r)\sum_{n > B_{1}}
				\frac{k^{m + n}}{m! n!}\EXP{G_{(1, 1)}(\mathfrak{X})^{m + n}}\r|\\
				&\leq 2 \l|\EXP{\exp\l( k G_{(1, 1)}(\mathfrak{X}) \r) \times \sum_{n > B_{1}}\frac{k^{n}}{n!}G_{(1, 1)}(\mathfrak{X})^{n}}\r|
				+ \EXP{\sum_{m > B_{1}} \sum_{n > B_{1}}\frac{k^{m + n}}{m! n!}G_{(1, 1)}(\mathfrak{X})^{m + n}}.
			\end{align}
			Using the Cauchy-Schwarz inequality, we also find that the first term is
			\begin{align}
				&\leq 2 \l(\EXP{\exp\l( 2k G_{(1, 1)}(\mathfrak{X}) \r)}\r)^{1/2}
				\times \l(\EXP{\l(\sum_{n > B_{1}}\frac{k^{n}}{n!}G_{(1, 1)}(\mathfrak{X})^{n}\r)^{2}}\r)^{1/2}\\
				&= 2 \l(\EXP{\exp\l( 2k G_{(1, 1)}(\mathfrak{X}) \r)}\r)^{1/2}
				\times \l(\sum_{m > B_{1}} \sum_{n > B_{1}}\frac{k^{m + n} }{m! n!}\EXP{G_{(1, 1)}(\mathfrak{X})^{m + n}}\r)^{1/2}.
			\end{align}
			Therefore, if $\I = 1$, we have
			\begin{align}
				\bigg|\sum_{1 \leq \ell \leq \I}E(\ell)\bigg|
				&\leq \l(2 \l(\EXP{\exp\l(2 k G_{(1, 1)}(\mathfrak{X}) \r)}\r)^{1/2}
				+ \l(\sum_{m > B_{1}} \sum_{n > B_{1}}\frac{k^{m + n}}{m! n!}\EXP{G_{(1, 1)}(\mathfrak{X})^{m + n}}\r)^{1/2} \r)\\
				\label{ES:I11}
				&\qquad\times \l(\sum_{m > B_{1}} \sum_{n > B_{1}}\frac{k^{m + n}}{m! n!}\EXP{G_{(1, 1)}(\mathfrak{X})^{m + n}}\r)^{1/2}.
			\end{align}
			It follows from \cref{nMG1} and the inequality $n! \geq (n / e)^{n}$ that
			\begin{align}
				&\sum_{m > B_{1}} \sum_{n > B_{1}}\frac{k^{m + n}}{m! n!}\EXP{G_{(1, 1)}(\mathfrak{X})^{m + n}}
				\leq \l( \sum_{n > B_{1}} (40 k \sqrt{\log{\log{\log{T}}}}/ \sqrt{n})^{n} \r)^{2}\\
				&\leq \l( \sum_{n > B_{1}} \l(\frac{130 k \sqrt{\log{\log{T}} \log{\log{\log{T}}}}}{\sqrt{\log{T}}}\r)^{n} \r)^{2}
				\leq \exp\l(- \frac{\log{T}}{\log{\log{T}}}\r)
			\end{align}
			under the assumptions $\I = 1$ and $T \geq e^{(1000 k \log{\k})^{2} } T_{0}$.
			As with \cref{EMI}, we also have
			\begin{align}
				\EXP{\exp\l(2 k G_{(1, 1)}(\mathfrak{X}) \r)}
				= e^{-k^{2}(\log{\log{\k}} + O(1))} (\log{T^{\delta_{1}}})^{k^{2}}.
			\end{align}
			Hence, we obtain
			\begin{align}
				\label{ES:I12}
				&\bigg|\sum_{1 \leq \ell \leq \I}E(\ell)\bigg|\\
				&\leq \l(2 e^{-k^{2}(\log{\log{\k}} + O(1))/2} (\log{T^{\delta_{1}}})^{k^{2}/2} + \exp\l(-\frac{\log{T}}{2\log{\log{T}}}\r) \r)
				\times \exp\l(-\frac{\log{T}}{2 \log{\log{T}}}\r),
			\end{align}
			which deduces \cref{EETI} if $\I = 1$ and $T \geq e^{(1000 k \log{\k})^{2}}T_{0}$.
			Next, we show \cref{EETI} by assuming $\I \geq 2$.
			This assumption leads to $\delta_{\I}^{-1} \geq (2 \delta_{\I - 1})^{-1} \geq e^{80 A(h, \theta) \k K} / 2$,
			$T \geq \exp(e^{80 A(h, \theta) \k K})$,
			and to $(L - i + 1)^{1/2} = \delta_{i}^{-1/16} \geq \delta_{\I}^{-1/16} \geq (\delta_{\I-1}/2)^{-1/16} \geq e^{5 A(h, \theta) \k K} / 2$
			for any $2 \leq i \leq \I$.
			Using \cref{nMG1}, we then find that
			\begin{align}
				\label{ES_E1g}
				|E(1)|
				&\leq \l\{2 e^{-k^{2}(\log{\log{\k}} + O(1))/2}(\log{T^{\delta_{1}}})^{k^{2}/2}
				+ \exp\l(-\frac{1}{2}\l(\frac{\log{T}}{\log{\log{T}}}\r)^{3/4}\r)\r\}\\
				&\qquad \times \exp\l(-\frac{1}{2}\l(\frac{\log{T}}{\log{\log{T}}}\r)^{3/4}\r)
				\times \prod_{2 \leq i \leq \I} \l(\sum_{n = 0}^{\infty}\frac{1}{(n!)^{1/2}}\l( \frac{60 k}{(L - i + 1)^{1/2}} \r)^{n}\r)^{2}.
			\end{align}
			similar to \cref{ES:I12}.
			In this case, the power $1$ of $\log{T} / \log{\log{T}}$ in \cref{ES:I12} is changed to $3/4$.
			This change comes from that we used the inequality $\I \leq L = (\log{T} / \log{\log{T}})^{1/8}$, but $\I = 1$ in \cref{ES:I12}.
			By simple calculations, the right hand side of \cref{ES_E1g} is
			\begin{align}
				\label{ESE1}
				\leq \exp\l( - \tfrac{1}{4}\l(\frac{\log{T}}{\log{\log{T}}}\r)^{3/4} \r) \times 2^{2\I} 
				\leq \exp\l( - \tfrac{1}{8}\l(\frac{\log{T}}{\log{\log{T}}}\r)^{3/4} \r)
			\end{align}
			As with \cref{ES:I11}, we also find that
			\begin{align}
				|E(\ell)|
				&\leq \EXP{\exp\l( 2k\sum_{1 \leq i \leq \ell - 1}G_{(i, \I)}(\mathfrak{X}) \r)}\\
				&\qquad \times \l(2 \l(\EXP{\exp\l(2 k G_{(\ell, \I)}(\mathfrak{X}) \r)}\r)^{1/2}
				+ \l(\sum_{m > B_{\ell}} \sum_{n > B_{\ell}}\frac{k^{m + n}}{m! n!}\EXP{G_{(\ell, \I)}(\mathfrak{X})^{m + n}}\r)^{1/2} \r)\\
				&\qquad\times \l(\sum_{m > B_{\ell}} \sum_{n > B_{\ell}}\frac{k^{m + n}}{m! n!}\EXP{G_{(\ell, \I)}(\mathfrak{X})^{m + n}}\r)^{1/2}\\
				&\qquad\times \prod_{\ell + 1 \leq i \leq \I}\sum_{m = 0}^{\infty}\sum_{n = 0}^{\infty}
				\frac{k^{m + n}}{m! n!}\EXP{|G_{(i, \I)}(\mathfrak{X})|^{m + n}}
			\end{align}
			for $2 \leq \ell \leq \I $. Moreover, we deduce from \cref{nMG1} that
			\begin{align}
				\sum_{m > B_{\ell}} \sum_{n > B_{\ell}}\frac{k^{m + n}}{m! n!}\EXP{G_{(\ell, \I)}(\mathfrak{X})^{m + n}}
				&\leq \l( \sum_{n > B_{\ell}}\frac{1}{(n!)^{1/2}}\l( \frac{60k}{(L - \ell + 1)^{1/2}} \r)^{n} \r)^{2}
				\leq e^{-B_{\ell}},
			\end{align}
			and that
			\begin{align}
				\prod_{\ell + 1 \leq i \leq \I}\sum_{m = 0}^{\infty}\sum_{n = 0}^{\infty}
				\frac{k^{m + n}}{m! n!}\EXP{|G_{(i, \I)}(\mathfrak{X})|^{m + n}}
				&\leq \l( \prod_{\ell + 1 \leq i \leq \I}\sum_{n = 0}^{\infty}\frac{1}{(n!)^{1/2}}\l( \frac{60k}{(L - \ell + 1)^{1/2}} \r)^{n} \r)^{2}\\
				&\leq 2^{2(\I - \ell)}
				\leq e^{2(\I - \ell)}.
			\end{align}
			Moreover, we can show that
			\begin{align}
				\EXP{\exp\l(2 k G_{(\ell, \I)}(\mathfrak{X}) \r)}
				= e^{-k^{2}(\log{\log{\k}} + O(1))}(\delta_{\ell} / \delta_{\ell - 1})^{k^{2}}
				= e^{-k^{2}(\log{\log{\k}} + O(1))}
				\ll 1,
			\end{align}
			and that
			\begin{align}
				\EXP{\exp\l( 2k\sum_{1 \leq i \leq \ell - 1}G_{(i, \I)}(\mathfrak{X}) \r)}
				= e^{-k^2(\log{\log{\k}} + O(1))}(\log{T^{\delta_{\ell - 1}}})^{k^{2}}
			\end{align}
			similar to \cref{EMI}.
			Therefore, we have
			\begin{align}
				|E(\ell)|
				\ll e^{-\frac{B_{\ell}}{2} + 2(\I - \ell)}e^{-k^2(\log{\log{\k}} + O(1))}(\log{T^{\delta_{\ell - 1}}})^{k^{2}}
			\end{align}
			for $2 \leq \ell \leq \I$.
			It holds by the definitions of $\delta_{j}$, $B_{j}$, and $\I$ that
			$B_{\ell} \geq 10^{-1} \delta_{\ell}^{-3/4} = 10^{-1} \delta_{\ell}^{-5/8} (L - \ell + 1) \geq 4(\I - \ell)$.
			Hence, we obtain
			\begin{align}
				\bigg| \sum_{1 \leq \ell \leq \I}E(\ell) \bigg|
				\ll \sum_{2 \leq \ell \leq \I}e^{-\frac{B_{\ell}}{4}}e^{-k^2(\log{\log{\k}} + O(1))}(\log{T^{\delta_{\ell - 1}}})^{k^{2}}
				\ll e^{-k^2(\log{\log{\k}} + O(1))}(\log{T^{\delta_{\I}}})^{k^{2}}.
			\end{align}
			This estimate and \cref{ESE1} imply \cref{EETI}.

			From \cref{SMEI}, \cref{EMI}, and \cref{EETI}, we have
			\begin{align}
				\label{BdI3}
				\prod_{1 \leq i \leq \I}\sum_{0 \leq n \leq B_{i}}\frac{(2k)^{n}}{n!}\EXP{G_{(i, \I)}(\mathfrak{X})^{n}}
				\ll e^{-k^{2}(\log{\log{\k}} + O(1))} (\log{T^{\delta_{\I}}})^{k^{2}}.
			\end{align}
			This implicit constant is absolute.
			Combining \cref{BdI2} with \cref{SMEI}, \cref{EMI}, and \cref{BdI3}, we complete the proof of \cref{lem_T}.
		\end{proof}

		\begin{proof}[Proof of \cref{lem_Sj}]
			First we evaluate $\meas \Sc(0)$.
			We have
			\begin{align}
				\meas \Sc(0)
				&\leq \sum_{1 \leq \ell \leq \I}\meas \mathcal{A}(1, \ell)^{c} + \meas \mathcal{B}(1)^{c}\\
				&= \sum_{1 \leq \ell \leq \I}\meas \mathcal{A}(1, \ell)^{c} + \meas([T, 2T] \setminus \mathcal{E}_{K}(T^{\delta_{1}}, T)).
			\end{align}
			By \cref{SLL} and the definition of $\mathcal{A}(1, \ell)$,
			we find for $N \in \ZZ_{\geq 1}$ with $N \leq \frac{\log{T}}{\log{\log{T}}}$ that
			\begin{align}
				&\meas\mathcal{A}(1, \ell)^{c}
				\leq \int_{T}^{2T}\l( 20 e^{2} \k \ell^{2} \delta_{1} G_{(1, \ell)}(t) \r)^{2N}dt\\
				&\leq (20 e^{2} \k \I^{2} \delta_{1})^{2N} \int_{T}^{2T}
				\bigg| \sum_{p \leq T^{\log{\log{T}} / \log{T}}} \frac{\phi_{j}(p)}{p^{1/2 + \i t}} + \frac{\psi_{j}(p)}{2p^{1 + 2 \i t}} \bigg|^{2N}dt\\
				&\ll T (25600 e^{4} \k^{2} \I^{4} \delta_{1}^{2} N \log{\log{\log{T}}})^{N}
				= T \l(\frac{25600 e^{4} \k^{2} \I^{4} N (\log{\log{T}})^{2} \log{\log{\log{T}}}}{(\log{T})^{2}}\r)^{N}.
			\end{align}
			If $\I = 1$, taking $N = \GS{\log{T} / \log{\log{T}}}$, we have
			\begin{align}
				\sum_{1 \leq \ell \leq \I}\meas\mathcal{A}(1, \ell)^{c}
				= \meas\mathcal{A}(1, 1)^{c}
				\ll T \l(\frac{25600 e^{4} \k^{2} (\log{\log{T}})^{2}}{\log{T}}\r)^{N}
				\leq T\exp\l(-\frac{\log{T}}{\log{\log{T}}}\r)
			\end{align}
			for $T \geq e^{(1200 k \log{\k})^{2}} T_{0}$.
			If $\I \geq 2$, we see from the definitions of $\delta_{j}$ and $\I$ that $T \geq \exp(e^{80 A(h, \theta) \k K})$,
			so we have
			\begin{align}
				\sum_{1 \leq \ell \leq \I}\meas\mathcal{A}(1, \ell)^{c}
				&\ll \I \cdot T \l(\frac{25600 e^{2} \k^{2} \I^{4} N (\log{\log{T}})^{2} \log{\log{\log{T}}}}{(\log{T})^{2}}\r)^{N}
				\leq T \l(\frac{N}{(\log{T})^{4/3}}\r)^{N}
			\end{align}
			for any large $T$ since $\I \leq L = (\log{T} / \log{\log{T}})^{1/8}$.
			Taking $N = \GS{\log{T} / \log{\log{T}}}$, we find that the last expression is $T \exp\l( - \frac{\log{T}}{\log{\log{T}}} \r)$.
			Combining this estimate and \cref{EXDZ}, we have $\meas\Sc(0) \ll T (\Phi(T) + 1) \exp\l(- \frac{\log{T}}{\log{\log{T}}}\r)$
			for $T \geq e^{(1200 k \log{\k})^{2}} T_{1}(K)$.

			Next, we prove \cref{lem_Sj_1}.
			Then we may assume that $\I \geq 2$.
			Write $\Sc(j) = \Sc_{1}(j) \cup \Sc_{2}(j)$ with
			\begin{align}
				\Sc_{1}(j)
				&\ceq \bigcap_{1 \leq i \leq j}\mathcal{A} \cap \mathcal{B}(j) \cap \bigcup_{j + 1 \leq \ell \leq \I} \mathcal{A}(j + 1, \ell)^{c},\\
				\Sc_{2}(j)
				&\ceq \bigcap_{1 \leq i \leq j}\mathcal{A} \cap \mathcal{B}(j) \cap \mathcal{B}(j + 1)^{c}.
			\end{align}
			It then follows that
			\begin{align}
				&\int_{\mathcal{S}(j)}\exp\l( 2k \sum_{p \leq T^{\delta_{j}}}a_{j}(p, t) \r)dt\\
				&\leq \l(\int_{\Sc_{1}(j)} + \int_{\Sc_{2}(j)}\r)
				\exp\l( 2k \sum_{p \leq T^{\delta_{j}}}a_{j}(p, t) \r)dt
				\eqc I_{1}(j) + I_{2}(j),
			\end{align}
			say.
			Put $B_{i} = 5^{-1}(j + 1 - i)^{-2} \delta_{i}^{-1}$ and $M = \lfloor (4 \delta_{j + 1})^{-1} \rfloor$.
			As with \cref{INE:EXTR}, we find for $t \in \mathcal{S}_{1}(j)$ that
			\begin{align}
				& \exp\l( 2k \sum_{p \leq T^{\delta_{j}}}a_{j}(p, t) \r)
				= \prod_{1 \leq i \leq j}\exp\l( 2k G_{(i, j)}(t) \r)\\
				& \leq \l( 20 e^{2} \k (\ell - j)^{2} \delta_{j + 1} G_{(j+1, \ell)}(t) \r)^{2 M} \prod_{1 \leq i \leq j}\exp\l( 2k G_{(i, j)}(t) \r)\\
				\label{BD_I1j1}
				& \leq C \l( \delta_{j + 1}^{3/5} G_{(j+1, \ell)}(t) \r)^{2 M}
				\prod_{1 \leq i \leq j}\l(\sum_{0 \leq n \leq B_{i}}\frac{(k G_{(i, j)}(t))^{n}}{n!}\r)^{2}
			\end{align}
			for some $j+1 \leq \ell \leq \I $ with some absolute constant $C > 0$.
			Therefore, we have
			\begin{align}
				I_{1}(j)
				\ll e^{- \frac{6}{5}M \log(1 / \delta_{j + 1})} \sum_{\ell = j + 1}^{\I}\int_{T}^{2T}G_{(j+1, \ell)}(t)^{2 M} \prod_{1 \leq i \leq j}
				\l(\sum_{0 \leq n \leq B_{i}}\frac{(k G_{(i, j)}(t))^{n}}{n!}\r)^{2}dt,
			\end{align}
			where $I_{1}(j)$ is the integral we evaluate.
			As with \cref{BdI2}, we can obtain
			\begin{align}
				I_{1}(j)
				& \ll e^{- \frac{6}{5} M \log(1 / \delta_{j + 1})} T \sum_{\ell = j + 1}^{\I}\EXP{G_{(j + 1, \ell)}(\mathfrak{X})^{2 M}}
				\prod_{1 \leq i \leq j}\sum_{0 \leq m \leq B_{i}}\sum_{0 \leq n \leq B_{i}}\frac{k^{m + n}}{m! n!}\EXP{G_{(i, j)}(\mathfrak{X})^{m + n}}\\
				\label{BdI1j2}
				&\qquad + e^{2k j}T^{1/2},
			\end{align}
			and further
			\begin{align}
				\label{BdI1j3}
				\prod_{1 \leq i \leq j}\sum_{0 \leq m \leq B_{i}}\sum_{0 \leq n \leq B_{i}}\frac{k^{m + n}}{m! n!}\EXP{G_{(i, j)}(\mathfrak{X})^{m + n}}
				\ll e^{-k^{2}(\log{\log{\k}} + O(1))} (\log{T^{\delta_{j}}})^{k^{2}}
			\end{align}
			similar to \cref{BdI3}.
			We can also prove that
			\begin{align}
				\sum_{\ell = j + 1}^{\I}\EXP{G_{(j + 1, \ell)}(\mathfrak{X})^{2 M}}
				&\leq (\I - j)\sqrt{(2M)!}\l(\frac{60}{(L - j)^{1/2}}\r)^{2M}\\
				&\leq (2 M)^{M} \frac{3600 (\I - j)}{L - j}
				\ll e^{M \log{(1 / \delta_{j + 1})}}
			\end{align}
			similar to \cref{nMG1}.
			Applying this estimate and \cref{BdI1j3} to \cref{BdI1j2} and using the inequality $\delta_{j+1} \leq (1 + 100^{-1})\delta_{j}$, we have
			\begin{align}
				I_{1}(j)
				&\ll e^{- \frac{1}{5} M \log(1 / \delta_{j + 1})}
				e^{-k^{2}(\log{\log{\k}} + O(1))} T (\log{T^{\delta_{j}}})^{k^{2}}\\
				\label{ES_I1j}
				&\ll e^{- \frac{1}{21} \delta_{j}^{-1} \log(1 / \delta_{j + 1})}
				e^{-k^{2}(\log{\log{\k}} + O(1))} T (\log{T^{\delta_{j}}})^{k^{2}}.
			\end{align}

			Next, we consider $I_{2}(j)$.
			It follows from the definition of $\Sc_{2}(j)$ and H\"older's inequality that
			\begin{multline}
				I_{2}(j)
				\leq \l( \meas\mathcal{B}(j + 1)^{c} \r)^{1/a}\\
				\times \l( \int_{\bigcap_{1 \leq i \leq j}\mathcal{A}(i, j) \cap \mathcal{B}(j)}
				\exp\l( 2k b \Re \sum_{p \leq T^{\delta_{j}}}\l(\frac{\phi_{j}(p)}{p^{1/2}}p^{-\i t} + \frac{\psi_{j}(p)}{2 p}p^{-2 \i t} \r) \r)dt \r)^{1/b}
			\end{multline}
			for $a, b > 1$ with $\frac{1}{a} + \frac{1}{b} = 1$.
			By the fact $\mathcal{B}(j + 1) = \mathcal{E}_{K}(T^{\delta_{j + 1}}, T)$ and \cref{EXDZ},
			we have $\meas\mathcal{B}(j + 1)^{c} \ll T K e^{K} \frac{\Phi(T)}{\delta_{j + 1} \log{T}} \exp\l( - \lam K \delta_{j + 1}^{-1} \r)
				\leq T K e^{K} \frac{\Phi(T)}{\log{T}} \exp\l( - (\lam K - 1) \delta_{j + 1}^{-1} \r)$.
			As with \eqref{BdI2} and \eqref{BdI3}, we can prove that
			\begin{align}
				&\l(\int_{\bigcap_{1 \leq i \leq j}\mathcal{A}(i, j) \cap \mathcal{B}(j)}
				\exp\l( 2k b \Re \sum_{p \leq T^{\delta_{j}}}\l(\frac{\phi_{j}(p)}{p^{1/2}}p^{-\i t}
				+ \frac{\psi_{j}(p)}{2 p}p^{-2 \i t} \r) \r)dt\r)^{1/b}\\
				&\ll e^{-k^{2} b (\log{\log{(\k b)}} + O(1))} T^{1/b} (\log{T^{\delta_{j}}})^{k^{2} b}.
			\end{align}
			Hence, we have
			\begin{align}
				I_{2}(j)
				\ll \l( K e^{K}\frac{\Phi(T)}{\log{T}} \r)^{1/a} e^{- (\lam K - 1)\delta_{j + 1}^{-1} / a}
				e^{-k^{2} b (\log{\log{(\k b)}} + O(1))} T (\log{T^{\delta_{j}}})^{k^{2} b}.
			\end{align}
			Combing \cref{ES_I1j} with this, we complete the proof.
		\end{proof}

\section{\textbf{Proof of \cref{MTUBMZ}}}
	Let $0 < \e \leq \frac{1}{100}$, and let $K = \e^{-2} K_{0}$ with $K_{0}$ a large constant.
	We divide the integral as
	\begin{align}
		&\int_{T}^{2T}\exp\l( 2k \Re e^{-\i \theta} \log{\zeta(\tfrac{1}{2} + \i t)} \r)dt\\
		&\leq \int_{\mathcal{T}}\exp\l( 2k \Re e^{-\i \theta} \log{\zeta(\tfrac{1}{2} + \i t)} \r)dt
		+ \sum_{j = 0}^{\mathcal{I} - 1}\int_{\Sc(j)}\exp\l( 2k \Re e^{-\i \theta} \log{\zeta(\tfrac{1}{2} + \i t)} \r)dt.
	\end{align}
	When $t \in \mathcal{T}$, it follows from \cref{AFLZ} that
	\begin{align}
		&\Re e^{-\i \theta} \log{\zeta(\tfrac{1}{2} + \i t)}
		\leq \sum_{p \leq T^{\delta_{\I}}}a_{\I}(p, t) + \frac{K + 30}{2 A(h, \theta)} \delta_{\I}^{-1} + O(\e^{-2})
	\end{align}
	since $\frac{2}{5} < h < \frac{3}{4}$.
	By this inequality and \cref{lem_T}, we obtain
	\begin{align}
		\int_{\mathcal{T}}\exp\l( 2k \Re e^{-\i \theta} \log{\zeta(\tfrac{1}{2} + \i t)} \r)dt
		\leq \exp(C \e^{-2} \exp(C \e^{-2} \k)) T (\log T)^{k^{2}}
	\end{align}
	for some absolute constant $C > 0$.
	Similarly, we can obtain, by \cref{lem_Sj_1},
	\begin{align}
		&\int_{\Sc(j)}\exp\l( 2k \Re e^{-\i \theta} \log{\zeta(\tfrac{1}{2} + \i t)} \r)dt\\
		&\ll e^{-(\frac{1}{21}\log{(1 / \delta_{j + 1})} - k K / A(h, \theta) - \e^{2})\delta_{j}^{-1}} T (\log{T})^{k^{2}}\\
		&\qquad + \l( K e^{K} \frac{\Phi(T)}{\log{T}} \r)^{1/a}
		e^{- (\lam K - 1)\delta_{j + 1}^{-1}/a + k (K + 30) \delta_{j}^{-1} / A(h, \theta)} T (\log{T})^{k^{2} b}
	\end{align}
	for $1 \leq j \leq \I - 1$ and $a, b > 1$ with $\frac{1}{a} + \frac{1}{b} = 1$.
	Since $K = \e^{-2}$, we have $\delta_{j + 1}^{-1} = (1 + O(\e^{2}))\delta_{j}^{-1}$ by the definition of $\delta_{j}$.
	Taking $a = A(h, \theta) \lam / (k + \e)$, we find that $b = 1 + \frac{k + \e}{A(h, \theta)\lam - (k + \e)}$, and that
	\begin{align}
		&\sum_{j = 1}^{\I - 1}\int_{\Sc(j)}\exp\l( 2k \Re e^{-\i \theta} \log{\zeta(\tfrac{1}{2} + \i t)} \r)dt\\
		&\ll_{\e} e^{-e^{k}} T (\log{T})^{k^{2}} + T (\log{T})^{k^{2}(1 + \frac{k + \e}{A(h, \theta)\lam  - (k + \e)})}
		(\Phi(T) / \log{T})^{\frac{k + \e}{A(h, \theta) \lam}}
	\end{align}
	for $0 \leq k < A(h, \theta) \lam - \e$.

	Finally, we consider the integral on $\Sc(0)$.
	It follows from H\"older's inequality that
	\begin{align}
		\label{pMTUBMZ1}
		&\int_{\Sc(0)}\exp\l( 2k \Re e^{-\i \theta} \log{\zeta(\tfrac{1}{2} + \i t)} \r)dt\\
		&\leq \l( \meas\Sc(0) \r)^{\frac{\e}{1+\e}}
		\times \l( \int_{T}^{2T}\exp\l( 2k(1+\e) \Re e^{-\i \theta} \log{\zeta(\tfrac{1}{2} + \i t)} \r)dt \r)^{\frac{1}{1+\e}}.
	\end{align}
	We see that
	\begin{align}
		&\int_{T}^{2T}\exp\l( 2k (1 + \e) \Re e^{-\i \theta} \log{\zeta(\tfrac{1}{2} + \i t)} \r)dt\\
		&= 2 k (1 + \e) T\int_{-\infty}^{\infty}e^{2k (1 + \e) V}\PP_{T}(\Re e^{-\i \theta} \log{\zeta(\tfrac{1}{2} + \i t)} > V)dV.
	\end{align}
	We find, using \cref{UBDRZ}, that
	\begin{align}
		&\int_{-\infty}^{\infty}e^{2k (1 + \e) V}\PP_{T}(\Re e^{-\i \theta} \log{\zeta(\tfrac{1}{2} + \i t)} > V)dV\\
		&\ll \int_{-\infty}^{\infty}e^{2k (1 + \e) V}\exp\l(- \frac{\e^{4} V^{2}}{4 e \log{\log{T}}} \r) dV
		+ \int_{-\infty}^{\infty}e^{2k (1 + \e) V} \exp\l( - \e^{2} A(h, \theta) V \log{V} \r)dV\\
		&\qquad+ \int_{-\infty}^{\infty}e^{2k (1 + \e) V} e^{K}\frac{\Phi(T)}{\log{T}}\exp\l( - (1 - \e^{2}) 2 A(h, \theta) \lam V \r)dV.
	\end{align}
	A simple calculation shows that the second term is $\leq e^{O(\e^{-2} k)}$, and that the third term is $\leq e^{O(\e^{-2})}\Phi(T) / \log{T}$
	when $k \leq A(h, \theta) \lam - \e$. Also, the first term is
	\begin{align}
		\leq \int_{-\infty}^{\infty}e^{4kV}\exp\l(- \frac{\e^{4} V^{2}}{4 e \log{\log{T}}} \r)dV
		&= (\log{T})^{16 e \e^{-4} k^{2}} \int_{-\infty}^{\infty}\exp\l(- \frac{\e^{4} V^{2}}{4 e \log{\log{T}}} \r)dV\\
		&\ll (\log{T})^{16 e \e^{-4} k^{2}} \e^{-2} \sqrt{\log{\log{T}}}.
	\end{align}
	Applying these estimates and \cref{lem_Sj} to \cref{pMTUBMZ1}, we have
	\begin{align}
		&\int_{\Sc(0)}\exp\l( 2k \Re e^{-\i \theta} \log{\zeta(\tfrac{1}{2} + \i t)} \r)dt\\
		&\ll T (\Phi(T) + 1)^{\frac{\e}{1 + \e}}\exp\l(-\frac{\e}{1 + \e}\frac{\log{T}}{\log{\log{T}}}\r)\\
		&\qquadf\times \l( e^{O(\e^{-2})} \frac{\Phi(T)}{\log{T}} + \e^{-2}(\log{T})^{16 e \e^{-4} k^{2}} \sqrt{\log{\log{T}}} \r)^{\frac{1}{1 + \e}}\\
		&\leq T \l(\frac{\Phi(T)}{\log{T}} + 1\r)
	\end{align}
	when $T \geq e^{C (k \log(k + 3))^{2}} T_{0}(\e)$ with $C$ some large constant depending only on $\e$.
	This completes the proof of \cref{MTUBMZ}.
	\qed

	\begin{acknowledgment*}
		The author was supported by Grant-in-Aid for JSPS Fellows (Grant Number 21J00425).
	\end{acknowledgment*}


\begin{thebibliography}{99}

		\bibitem{Co1989} J. B. Conrey, At least two fifths of the zeros of the Riemann zeta-function are on the critical line,
		\textit{Bull Amer. Math. Soc.} \textbf{20} (1989), no.1, 79--81.

		\bibitem{GHK2007} S. M. Gonek, C. P. Hughes, and J. P. Keating,
		A hybrid Euler-Hadamard product for the Riemann zeta-function,
		\textit{Duke Math. J.} no.3 \textbf{136} (2007), 507--549.

		\bibitem{HL1918} G. H. Hardy and J. E. Littlewood, Contributions to the theory of the Riemann zeta-function
		and the theory of the distribution of primes, \textit{Acta Arith.} \textbf{41} (1918), 119--196.

		\bibitem{H2013} A. J. Harper, Sharp conditional bounds for moments of the Riemann zeta function,
		preprint, \texttt{arXiv:1305.4618}.

		\bibitem{HRS2019} W. Heap, M. Radziwi\l\l, and K. Soundararajan,
		Sharp upper bounds for fractional moments of the Riemann zeta-function,
		\textit{Quart. J. Math.} \textbf{70} (2019), 1387--1396.

		\bibitem{In1926} A. E. Ingham, Mean-value theorems in the theory of the Riemann zeta function, \textit{Proc. Lond. Math. Soc.}
		\textbf{27} (1926), 273--300.

		\bibitem{In1937} A. E. Ingham, On the difference between consecutive primes, \textit{Quart. J. Math. Oxford Ser.} \textbf{1} (1937), 255--266.


		\bibitem{In1940} A. E. Ingham, On the estimation of $N(\s, T)$, \textit{Quart. J. Math. Oxford Ser.} \textbf{11} (1940), 291--292.

		\bibitem{II2019} S. Inoue, On the logarithm of the Riemann zeta-function and its iterated integrals,
		preprint, \texttt{arXiv:1909.03643}.

		\bibitem{IL2021} S. Inoue and J. Li, Joint value distribution of $L$-functions on the critical line,
		preprint, \texttt{arXiv:2102.12724}.

		\bibitem{KS2000} J. P. Keating and N. C. Snaith, Random matrix theory and $\zeta(1/2 + \i t)$, \textit{Comm. Math. Phys.} \textbf{214} (2000), 91--110.

		\bibitem{MV} H. L. Montgomery and R. C. Vaughan, \textit{Multiplicative Number Theory I. Classical Theory},
		Cambridge Studies in Advanced Mathematics, Cambridge University press, 2007.

		\bibitem{Na2020} J. Najnudel, Exponential moments of the argument of the Riemann zeta function on the critical line,
		\textit{Mathematika} \textbf{66} (2020), 612--621.

		\bibitem{SCR} A. Selberg, Contributions to the theory of the Riemann zeta-function,
		Avhandl. Norske Vid.-Akad. Olso I. Mat.-Naturv. Kl., no.1;
		Collected Papers, Vol. 1, New York: Springer Verlag. 1989, 214--280.

		\bibitem{SM2009} K. Soundararajan, Moments of the Riemann zeta-function,
		\textit{Ann. of Math.} (2) \textbf{170} (2009), 981--993.
	\end{thebibliography}
\end{document}